\documentclass[11pt]{article}
\usepackage{amssymb,float,amsthm,hyperref,xcolor,amsmath,tipa}
\hypersetup{
    colorlinks,
    citecolor=blue,
    filecolor=blue,
    linkcolor=blue,
    urlcolor=blue
}
\usepackage{witharrows}
\usepackage[noabbrev,capitalize,nameinlink]{cleveref}
\newtheorem{theorem}{Theorem}[section]
\newtheorem{lemma}[theorem]{Lemma}
\newtheorem{corollary}[theorem]{Corollary}
\newtheorem{proposition}[theorem]{Proposition}

\theoremstyle{definition}
\newtheorem{remark}[theorem]{Remark}
\newtheorem{example}[theorem]{Example}
\newtheorem{definition}[theorem]{Definition}
\newtheorem{convention}[theorem]{Convention}

\newcommand{\rsim}{\stackrel{r}{\sim}}

\numberwithin{equation}{section}
\usepackage[margin=1in]{geometry}
\counterwithin{figure}{section}

\begin{document}
\title{Enumerating pattern-avoiding permutations by leading terms}

\author{\"{O}mer E\u{g}ecio\u{g}lu$^1$\qquad Collier Gaiser$^{2}$\qquad Mei Yin$^2$\\ \\
\small $^1$Department of Computer Science \\
\small University of California Santa Barbara \\
\small Santa Barbara, CA 93106\\
\small \tt omer@cs.ucsb.edu
\\
\\
\small $^2$Department of Mathematics \\
\small University of Denver\\
\small Denver, CO 80208\\
\small \tt collier.gaiser@du.edu, mei.yin@du.edu
}

\date{}

\maketitle

\begin{abstract}
The number of $123$-avoiding permutations on $\{1,2,\ldots,n\}$ with a fixed leading 
term is counted by the ballot numbers. The same holds for $132$-avoiding permutations.
These results were proved by Miner and Pak 
using the Robinson–Schensted–Knuth (RSK) correspondence to connect permutations with Dyck paths. 
In this paper, we first provide an alternate proof of these enumeration results via a
direct counting argument.
We then study the number of pattern-avoiding permutations 
with a fixed prefix of length $t\geq 1$, generalizing the $t=1 $ case.
We find exact expressions for single and pairs of patterns of 
length three as well as the pair $3412$ and $3421$. 
These expressions depend on $t$, the extrema, and the order statistics.
We also define $r$-Wilf equivalence for permutations with a 
single fixed leading term $r$, and classify the $r$-Wilf-equivalence classes for both 
classical and vincular patterns of length three.
\end{abstract}


\section{Introduction}
Let $A\subseteq\mathbb{N}:=\{1,2,\ldots\}$ be a finite set. We use $|A|$ to denote the number of elements in $A$. A permutation $\tau$ on $A$ is a sequence $(\tau(1),\tau(2),\ldots,\tau(|A|))$ of length $|A|$ consisting of distinct numbers in $A$. 
 When $A\subseteq\{1,2,\ldots,9\}$ or when there is no confusion, we 
simply write a permutation/sequence without commas or 
parentheses in single-line notation. The leading $t\geq 1$ terms of a 
sequence is a prefix of length $t$ 
of the sequence.
When $t=1$, the prefix is also called the \textit{leading term} of the sequence.
We use $S_A$ to denote the set of permutations on $A$. When $A=[n]:=\{1,2,\ldots,n\}$ for some $n\in\mathbb{N}$, we 
write $S_n$ for $S_{[n]}$.

For any $\tau\in S_n$ and $\sigma\in S_k$, if there exist $1\leq i_1<i_2<\cdots<i_k\leq n$ such that for all $1\leq a<b\leq k$, $\tau(i_a)<\tau(i_b)$ if and only if $\sigma(a)<\sigma(b)$, 
then we say that $\tau$ \textit{contains $\sigma$ as a pattern} and 
that $(\tau(i_1),\tau(i_2)\ldots,\tau(i_k))$ is a $\sigma$ pattern.
A permutation
$\tau$ \textit{avoids $\sigma$} if $\tau$ does not contain $\sigma$ as a pattern. For example, the permutation $\tau=12453\in S_5$ contains the pattern $132$ because $\tau(1)\tau(3)\tau(5)=143$ is a $132$ pattern; however, $\tau$ avoids the pattern $321$. 
For any $m,n,k\in\mathbb{N}$ and $\sigma_1,\sigma_2,\ldots,\sigma_m\in S_k$, we use $S_n(\sigma_1,\sigma_2,\ldots,\sigma_m)$ to denote the set of permutations $\tau\in S_n$ such that $\tau$ avoids 
all of the patterns $\sigma_1,\sigma_2,\ldots,\sigma_m$. 

The interest in the study of 
pattern avoidance can be traced back to 
stack-sortable permutations in computer science \cite[Section 2.1]{Kitaev2011}. 
One of the earliest results is the enumeration of permutations avoiding $\sigma\in S_3$, 
i.e., patterns of length three. D. Knuth \cite{Knuth1973} proved that the number 
of permutations in $S_n$ avoiding any given pattern of length three is 
enumerated by the Catalan numbers $C_n$
(see also \cite[Theorem 4.7]{Bona2022}).

\begin{theorem} {\rm \cite[p. 238]{Knuth1973}} \label{Theorem:ClassicalSingle3}
For all $n\geq1$ and $\sigma\in S_3$, we have
\[
|S_n(\sigma)|=C_n=\frac{1}{n+1} \binom{2n}{n}.
\]
\end{theorem}

For $r,n\in\mathbb{N}$ with $r\leq n$, let $S_{n,r}$ denote the set 
of permutations $\tau\in S_n$ with $\tau(1)=r$.
It is clear that $|S_{n,r}|=(n-1)!$ for all $r\in[n]$. 
While studying the shapes of pattern-avoiding permutations, Miner and Pak \cite{MinerPak2014} proved that $S_{n,r}(123)$ and $S_{n,r}(132)$ are enumerated by the 
ballot numbers (see, for example Aval \cite{Aval2008} for more details on the ballot numbers).

\begin{theorem} {\rm \cite[Lemmas 4.1 and 5.2]{MinerPak2014}}\label{Theorem:ClassicalLeading123AND132}
    For all $1\leq r\leq n$, we have \[|S_{n,r}(123)|=|S_{n,r}(132)|=b_{n,r}=\frac{n-r+1}{n+r-1}\binom{n+r-1}{n}.\]
\end{theorem}

Miner and Pak \cite{MinerPak2014} proved \cref{Theorem:ClassicalLeading123AND132} via a bijection between $S_{n,r}(123)$ (respectively, $S_{n,r}(132)$) and certain types of Dyck paths 
using the Robinson–Schensted–Knuth (RSK) correspondence. In this paper, we prove \cref{Theorem:ClassicalLeading123AND132} using 
a direct counting argument. By the classical bijection between $S_n(123)$ and $S_n(132)$ \cite[Lemma 4.4]{Bona2022} which preserves the 
leading term, one only needs to prove that $|S_{n,r}(123)|=b_{n,r}$. We achieve this 
by utilizing a result of Simion and Schmidt \cite[Lemma 2]{SimionSchmidt1985} on the number of $123$-avoiding permutations with a 
fixed decreasing prefix.

It is natural to consider the case in which more than one leading term of the permutation is fixed.
Motivated by this general case, we study pattern-avoiding permutations 
with a fixed prefix $(c_1,c_2,\ldots,c_t)$ of length $t\geq 1$. Here the $t=1$ instance corresponds to 
the case studied by Miner and Pak.

\begin{definition}
    For any $n,m\in\mathbb{N}$, distinct integers $c_1,c_2,\ldots,c_t\in[n]$, and permutation patterns $\sigma_1,\sigma_2,\ldots,\sigma_m$, we use $S_{n,(c_1,c_2,\ldots,c_t)}$ to denote the set of permutations $\tau\in S_n$ such that $(\tau(1),\tau(2),\ldots,\tau(t))=(c_1,c_2,\ldots,c_t)$; and we use $S_{n,(c_1,c_2,\ldots,c_t)}(\sigma_1,\sigma_2,\ldots,\sigma_m)$ to denote the set of permutations $\tau\in S_{n,(c_1,c_2,\ldots,c_t)}$ such that $\tau$ 
avoids all of the patterns $\sigma_1,\sigma_2,\ldots,\sigma_m$.
\end{definition}

\begin{convention}\label{Convention:AvoidanceForLeadingTerms}
Unless otherwise specified, for $S_{n,(c_1,c_2,\ldots,c_t)}(\sigma_1,\sigma_2,\ldots,\sigma_m)$, 
we assume that the fixed prefix 
$(c_1,c_2,\ldots,c_t)$ itself avoids all of the patterns $\sigma_1,\sigma_2,\ldots,\sigma_m$, $n\geq3$, and $1\leq t<n$.
\end{convention}

We first show that the cardinality of 
$S_{n,(c_1,c_2,\ldots,c_t)}(\sigma)$ can be determined exactly for all $\sigma\in S_3$. For $\sigma\in\{123,132,321,312\}$, if $|S_{n,(c_1,c_2,\ldots,c_t)}(\sigma)|\neq0$, then $S_{n,(c_1,c_2,\ldots,c_t)}(\sigma)$ is enumerated by ballot numbers. This is because, as we will show later in the proof of \cref{Theorem:123LeadingTerms}, there is a natural bijection between $S_{n,(c_1,c_2,\ldots,c_t)}(\sigma)$ and $S_{n-t+1,r}(\sigma)$ for some $r$ 
which depends on $\{c_1,c_2,\ldots,c_t\}$. For $\sigma\in\{213,231\}$, if $|S_{n,(c_1,c_2,\ldots,c_t)}(\sigma)|\neq0$, then $|S_{n,(c_1,c_2,\ldots,c_t)}(\sigma)|$ is equal to a product of Catalan numbers. This is because, for all $\tau\in S_{n,(c_1,c_2,\ldots,c_t)}(\sigma)$, the 
order statistics of $\{c_1,c_2,\ldots,c_t\}$ determine 
a `block' structure for the suffix $(\tau(t+1),\tau(t+2),\ldots,\tau(n))$.

We then show that for all pairs of patterns of length three, 
the number of permutations avoiding these patterns and with 
a fixed prefix can also be determined exactly. The expressions for pairs of patterns 
of length three depend on the extrema, the order statistics, and the 
length of the prefix.
We also determine the cardinality of $S_{n,(c_1,c_2,\ldots,c_t)}(3412,3421)$. As in the 
case of $S_{n,(c_1,c_2,\ldots,c_t)}(231)$, if $|S_{n,(c_1,c_2,\ldots,c_t)}(3412,3421)|\neq 0$, 
then for all $\tau\in S_{n,(c_1,c_2,\ldots,c_t)}(3412,3421)$, the suffix $(\tau(t+1),\tau(t+2),\ldots,\tau(n))$ 
has a 
`block' structure. However, since the patterns $3412$ and $3421$ are of length four, 
some overlap between these blocks are possible. The cardinality 
of $S_{n,(c_1,c_2,\ldots,c_t)}(3412,3421)$ is related to the large Schr\"{o}der numbers, 
and our treatment
incidentally provides a new combinatorial proof of 
a recurrence relation for large Schr\"{o}der numbers $\mathbb{S}_n$ \cite[p.~446]{Bona2022}. 

In addition to exact enumeration, the notion of Wilf equivalence classes are also an important topic in pattern 
avoidance. Two permutation patterns $\sigma$ and $\sigma'$ are said to 
be {\em Wilf equivalent}, denoted $\sigma\sim\sigma'$, if $|S_n(\sigma)|=|S_n(\sigma')|$ for all $n\in\mathbb{N}$. By \cref{Theorem:ClassicalSingle3}, all permutation patterns of length three are Wilf equivalent: $123\sim132\sim213\sim231\sim312\sim321$. In other words, there is only one Wilf-equivalence class for permutation patterns of length three. For patterns of length four, it is known that there are three Wilf-equivalence classes \cite[p.~158]{Bona2022}.

We consider a similar Wilf-equivalence concept for permutations with a fixed leading term. 
For a 
fixed $r\in\mathbb{N}$, two patterns $\sigma$ and $\sigma'$ are 
called {\em $r$-Wilf equivalent} if $|S_{n,r}(\sigma)|=|S_{n,r}(\sigma')|$ for all 
$n\geq r$. We 
write $\sigma\rsim\sigma'$ if $\sigma$ and $\sigma'$ are $r$-Wilf equivalent. As an example, two patterns $\sigma$ and $\sigma'$ are $2$-Wilf 
equivalent, denoted $\sigma\stackrel{2}{\sim}\sigma'$, if for all $n\geq2$, $|S_{n,2}(\sigma)|=|S_{n,2}(\sigma')|$. We show that there are two $1$-Wilf-equivalence classes 
for patterns of length three, 
$123\stackrel{1}{\sim} 132$ and $321\stackrel{1}{\sim}312\stackrel{1}{\sim}213\stackrel{1}{\sim}231$; and, for all $r\geq2$, there are three $r$-Wilf-equivalence classes for patterns of length three, $213\rsim231$, $123\rsim132$, and $321\rsim312$. We also show that for all $r\geq5$, there are nine $r$-Wilf-equivalence classes for vincular patterns of length three as studied in \cite{BabsonSteingrimsson2000,Claesson2001}.
\\

This paper is organized as follows. In \cref{Section:Preliminaries}, we define basic 
concepts and state our preliminary results. We then provide a new proof of \cref{Theorem:ClassicalLeading123AND132} in \cref{Section:NewProof}. In \cref{Section:Single3}, we present results on the number of permutations 
with a fixed prefix that avoid a single pattern of length three. 
The case for the avoidance of pairs of patterns of length three 
is presented in \cref{Section:Pair3}. Permutations avoiding both 3412 and 3421 are then studied in \cref{Section:Pair4}. In \cref{Section:LeadingWilf}, we classify $r$-Wilf-equivalence classes for classical and vincular patterns of length three.


\section{Preliminaries}\label{Section:Preliminaries}
\begin{definition}\label{Definition:Complements}
    For a permutation $\tau\in S_n$, the {\em complement} $\tau^c$ of $\tau$ is the permutation in $S_n$ defined by setting 
$\tau^c(i)=n+1-\tau(i)$.
\end{definition}
The following result relates permutations avoiding 
certain patterns with those permutations avoiding the complement of 
these patterns. Since the proof is elementary, we state it without proof.

\begin{lemma}\label{lemma:complements}
Let $t$, $n$, $m$, and $k$ be positive integers with $t,k\leq n$, $\sigma_1,\sigma_2,\dots,\sigma_m\in S_k$ permutation patterns, and $c_1,c_2,\ldots,c_t\in[n]$. Then we have \[|S_{n,(c_1,c_2,\ldots,c_t)}(\sigma_1,\sigma_2,\ldots,\sigma_m)|=|S_{n,(n+1-c_1,n+1-c_2,\ldots,n+1-c_t)}(\sigma_1^c,\sigma_2^c\ldots,\sigma_m^c)|.\]
\end{lemma}

\begin{definition}
    Let $A$ and $B$ be two finite subsets of $\mathbb{N}$ with $A\subseteq B$, $\sigma\in S_A$, and $\tau\in S_B$. We say that $\sigma$ is a 
{\em subpermutation} of $\tau$ on $A$ if 
there exist indices $1 \leq i_1<i_2<\cdots<i_{|A|} \leq |B|$ such that
    \[
    (\tau(i_1),\tau(i_2),\ldots,\tau(i_{|A|}))=(\sigma(1),\sigma(2),\ldots,\sigma(|A|)).
    \]
\end{definition}
For example, if  $\tau=543621\in S_6$, 
then $\sigma=462\in S_{\{2,4,6\}}$ is a subpermutation of $\tau$ on $\{2,4,6\}$.

\begin{definition}
    Suppose $\sigma$ is a permutation on a set $A$ and $\tau$ is a permutation on 
a set $B$ with $A\cap B=\emptyset$.
A {\em shuffle} of $\sigma$ and $\tau$ is a permutation $\alpha$ on $A\cup B$ 
such that $\sigma$ is a subpermutation of $\alpha$ on $A$ and $\tau$ 
is a subpermutation of $\alpha$ on $B$.
\end{definition} 
For example, if $A= \{4,5,7\}$, 
$B=\{1,3,6\}$,
$\sigma=457\in S_A$, and $\tau=631\in S_B$, then 
$\alpha=643571\in S_{A\cup B}$ and $\alpha'=456317\in S_{A\cup B}$ are shuffles of $\sigma$ and $\tau$. The 
following simple observation is crucial for our later derivations. We state it without proof.
\begin{lemma}
Suppose $A,B\subseteq \mathbb{N}$ with $|A\cap B|=\emptyset$, $|A|=k$, and $|B|=\ell$. If $\sigma\in S_A$ and $\tau\in S_B$, then the number of shuffles of $\sigma$ and $\tau$ is $\binom{k+\ell}{k}$.
\end{lemma}

We will use the following terminology. 
\begin{definition}
Let $A\subseteq\mathbb{N}$ be 
a finite set and $\tau\in S_{A}$. If $\tau(i)=a$, then we 
use $\mathcal{A}_\tau(a)=\{\tau(1),\tau(2),\ldots,\tau(i-1)\}$ to denote the set 
of {\em ancestors} of $a$ in $\tau$ and $\mathcal{D}_\tau(a)=\{\tau(i+1),\tau(i+2),\ldots,\tau(|A|)\}$ 
to denote the set of {\em descendants} of $a$ in $\tau$. 
\end{definition}

For example, if $\tau=2785\in S_{\{2,5,7,8\}}$, then $\mathcal{A}_\tau(8)=\{2,7\}$ 
and $\mathcal{D}_\tau(7)=\{5,8\}$.

\begin{definition}
Let $n\in\mathbb{N}$ and $A\subseteq[n]$. The \textit{standardization} of $\tau=(\tau(1),\tau(2),\ldots,\tau(|A|))\in S_A$ 
is the permutation $s(\tau)\in S_{|A|}$ obtained by replacing the $i$th smallest entry in $\tau$ 
with $i$ for all $i$.
\end{definition}

For example, the standardization of $567832\in S_{\{2,3,5,6,7,8\}}$ is $345621\in S_6$. We include below a simple observation concerning standardization. 
\begin{lemma}\label{Lemma:StandardizationPreservesPatternAvoidance}
    Let $n\in\mathbb{N}$, $A\subseteq[n]$, and $\tau\in S_A$. If $\tau$ avoids a pattern $\sigma$, 
then $s(\tau)$ also avoids the 
pattern $\sigma$.
\end{lemma}

\begin{definition}
    Let $n,n'\in\mathbb{N}$ with $n\leq n'$, $A\subseteq[n']$ with $|A|=n$, and $\tau\in S_{n}$. Then the \textit{matching permutation} $\tau'$ of $\tau$ on $A$ is defined as follows: if $\tau(i)=j$ where $i,j\in\{1,2,\ldots,n\}$, then $\tau'(i)$ is the $j$th smallest integer in $A$.
\end{definition}

For example, with $n=3$, $n'=7$, and $A=\{2,4,7\}$, the matching permutation of $231\in S_3$ on $\{2,4,7\}$ is $472\in S_{\{2,4,7\}}$.

Notice that the matching permutation of a permutation also preserves pattern avoidance. \\

The first few Catalan numbers $C_n$, Bell numbers $B_n$, and large Schr\"{o}der 
numbers $\mathbb{S}_n$ are listed in Table \ref{first_table}
for later reference.
\begin{table}[H]
\centering
\begin{tabular}{|c||c|c|c|c|c|c|c|c|c|c|c|c|}
\hline
$n$&0&1&2&3&4&5&6&7&8&9&10&OEIS \cite{OEIS}\\\hline\hline$C_n$&1&1&2&5&14&42&132&429&1430&4862&16796&A000108\\\hline
$B_n$&1&1&2&5&15&52&203&877&4140&21147&115975&A000110\\\hline
$\mathbb{S}_n$&1&2&6&22&90&394&1806&8558&41586&206098&1037718&A006318\\\hline

\end{tabular}
\caption{$C_n$, $B_n$, and $\mathbb{S}_n$ for $n\leq 10$.}
\label{first_table}
\end{table}

We also need the following elementary results on the Catalan and the Bell numbers.
\begin{lemma}\label{Lemma:CatalanBell}
    For all $n\geq4$, we have $C_n<B_n$.
\end{lemma}
\begin{proof}
    It is well-known that $C_n$ counts the number of noncrossing partitions 
of $[n]$ and $B_n$ counts the total number of partitions of $[n]$. For these 
facts and the definitions of 
partitions and noncrossing partitions, see for example 
\cite[Section 1.1]{Mezo2020} and \cite{Simion2000}. For all $n\geq4$, the two subsets $\{2,4\}$ and $[n]\backslash\{2,4\}$ form a crossing partition of $[n]$ and hence $C_n<B_n$.
\end{proof}
\begin{lemma}\label{Lemma:Bell}
For all $n\geq3$, we have $B_n>2B_{n-1}$.
\end{lemma}
\begin{proof}
    The Bell numbers $B_n$ satisfy the following recurrence relation \cite[p. 49]{EgeciogluGarsia2021}:
    \[
    B_n=\sum_{k=0}^{n-1}\binom{n-1}{k}B_k.
    \]
Let $n\geq3$. Then we have
    \begin{align*}
\hspace{3cm}B_n=&~B_{n-1}+\binom{n-1}{n-2}B_{n-2}+\binom{n-1}{n-3}B_{n-3}+\cdots+\binom{n-1}{0}B_0\\>
&~B_{n-1}+\binom{n-2}{n-2}B_{n-2}+\binom{n-2}{n-3}B_{n-3}+\cdots+\binom{n-2}{0}B_0\\=&~B_{n-1}+\sum_{k=0}^{n-2}\binom{n-2}{k}B_k=2B_{n-1}.\hspace{6.5cm}\qedhere
\end{align*}
\end{proof}

\section{A new proof of \cref{Theorem:ClassicalLeading123AND132}}\label{Section:NewProof}
Let $1\leq r\leq n$. The well-known bijection between $S_n(123)$ and $S_n(132)$ by Simion and Schmidt \cite{SimionSchmidt1985} preserves the 
leading term. 
We refer to \cite[Lemma 4.4]{Bona2022} for a proof of this fact. Hence, we have $|S_{n,r}(123)|=|S_{n,r}(132)|$. So we only need to show that $|S_{n,r}(123)|=b_{n,r}$. 

First, consider $r=1$. If $\tau\in S_{n,1}(123)$, then $(\tau(2),\tau(3),\ldots,\tau(n))$ must be a decreasing sequence and hence $\tau=(1,n,n-1,\ldots,2)$. Therefore, we have $|S_{n,1}(123)|=1=b_{n,1}$.

Now suppose $r\geq2$. We will need the following definition.

\begin{definition}\label{Definition:AscendIn123Avoiding}
For all $n\in\mathbb{N}$, define $a_{n}(i):[n]\to[n]$ as follows:
\begin{itemize}
\item[(i)] for all $i<n$, let $a_{n}(i)$ be the number of permutations $\tau\in S_n(123)$ such that $i$ is the smallest index with $\tau(i)<\tau(i+1)$, and
\item[(ii)] set $a_{n}(n)=1$.
\end{itemize}
\end{definition}
That is, if $i<n$, then $a_n(i)$ is the number of permutations $\tau\in S_n$ avoiding $123$ such that $\tau(1)>\tau(2)>\cdots>\tau(i-1)>\tau(i)<\tau(i+1)$. 
If $i=n$, then $a_n(n)=1$ 
because there is exactly one decreasing sequence $(n,n-1,\ldots,2,1)\in S_n$. Simion and Schmidt \cite[Lemma 2]{SimionSchmidt1985} proved the following result for $a_n(i)$:

\begin{lemma}\label{Lemma:AscendIn123Avoiding} For all $1\leq i\leq n$,
\[
a_n(i)=\binom{2n-i-1}{n-1}-\binom{2n-i-1}{n}.
\]
\end{lemma}

Let $\mathcal{P}$ be a subset of $S_{n,r}$ such that every $\tau\in\mathcal{P}$ has the following properties:
\begin{itemize}
    \item[(i)] the subpermutation $\tau'$ of $\tau$ on $\{1,2,\ldots,r-1\}$ avoids $123$;
    \item[(ii)] the subpermutation $\tau''$ of $\tau$ on $\{r+1,r+2,\ldots,n\}$ is $(n,n-1,\ldots,r+1)$;
    \item[(iii)] if $r>2$ and $i<r-1$ is the smallest index with $\tau'(1)>\tau'(2)>\cdots>\tau'(i)<\tau'(i+1)$, then $\{r+1,r+2,\ldots,n\}\subseteq\mathcal{A}_\tau(\tau'(i+1))$.
\end{itemize}
For Property (iii), we do not impose any extra condition on the positions of $\{r+1,r+2,\ldots,n\}$ when $\tau'=(r-1,r-2,\ldots,1)$.

We first show that
\[
|\mathcal{P}|=\sum_{i=1}^{r-1}\binom{i+n-r}{i}a_{r-1}(i).
\]
Let $\tau\in\mathcal{P}$. If $r=2$, then the subpermutation of $\tau$ on $\{2,3,\ldots,n\}$ is $(2,n,n-1,\ldots,3)$ and hence there are $n-1$ possible locations for $1$. It is not hard to see that all these possible locations guarantee that $\tau$ avoids the pattern $123$. It follows that, in this case, we have $|\mathcal{P}|=n-1$. It is easy to see that $a_1(1)=1$ and hence $|\mathcal{P}|=\sum_{i=1}^{2-1}\binom{i+n-2}{i}a_1(i)=n-1$. If $r>2$, then for all $i\in\{1,2,\ldots,r-2\}$ and a fixed $\tau'$ with $\tau'(1)>\tau'(2)>\cdots>\tau'(i)<\tau'(i+1)$, the number of shuffles of $(\tau'(1),\tau'(2),\ldots,\tau'(i))$ and $\tau''=(n,n-1,\ldots,r+1)$ is $\binom{i+n-r}{i}$. If $\tau'=(r-1,r-2,\ldots,1)$, then the number of shuffles of $\tau'$ and $\tau''$ is $\binom{r-1+n-r}{r-1}$. As the number of such $\tau'$ is $a_{r-1}(i)$ for all $i\in\{1,2,\ldots,r-1\}$, we have $|\mathcal{P}|=\sum_{i=1}^{r-1}\binom{i+n-r}{i}a_{r-1}(i)$.

It remains to show that $S_{n,r}(123)=\mathcal{P}$. Let $\tau\in S_{n,r}(123)$, and let $\tau'$ be the subpermutation of 
$\tau$ on $\{1,2,\ldots,r-1\}$ and $\tau''$ be the subpermutation of 
$\tau$ on $\{r+1,r+2,\ldots,n\}$. Since $\tau$ avoids $123$, 
$\tau'$ avoids $123$ as well. We now show that $\tau''$ avoids $12$. 
If this is not the case then
there exist $a<b\leq n-r$ such that $\tau''(a)<\tau''(b)$. 
Since $\tau''(a)>r$, $r\tau''(a)\tau''(b)$ is a 123 pattern, and 
this is a contradiction. Therefore $\tau''=(n,n-1,\ldots,r+1)$. Now suppose $r>2$ and 
let $i<r-1$ be the smallest index such 
that $\tau'(1)>\tau'(2)>\cdots>\tau'(i)<\tau'(i+1)$. We still need to show that $\{r+1,r+2,\ldots,n\}\subseteq\mathcal{A}_\tau(\tau'(i+1))$. Suppose, by way of contradiction, $\{r+1,r+2,\ldots,n\}\not\subseteq\mathcal{A}_\tau(\tau'(i+1))$. Then there exists $a\in\{r+1,r+2,\ldots,n\}\cap\mathcal{D}_\tau(\tau'(i+1))$. 
Since
$\tau'(i)\tau'(i+1)a$ is a 123 pattern, this would be a contradiction. Hence, we have $S_{n,r}(123)\subseteq\mathcal{P}$.

Now let $\tau\in\mathcal{P}$. We need to show that $\tau\in S_{n,r}(123)$. Let $\tau'$ be the subpermutation of 
$\tau$ on $\{1,2,\ldots,r-1\}$ and $\tau''$ be the subpermutation of 
$\tau$ on $\{r+1,r+2,\ldots,n\}$. If $r>2$, then let $i<r-1$ be the smallest index such that $\tau'(1)>\tau'(2)>\cdots>\tau'(i)<\tau'(i+1)$; if $\tau'=(r-1,r-2,\ldots,1)$, then we set $i=r-1$. Since $\tau\in\mathcal{P}$, $\tau=(r,\alpha(1),\alpha(2),\ldots,\alpha(n-r+i),\tau'(i+1),\tau'(i+2),\ldots,\tau'(r-1))$ where $\alpha=(\alpha(1),\alpha(2),\ldots,\alpha(n-r+i))$ is a shuffle of $(\tau'(1),\tau'(2),\ldots,\tau'(i))$ and $\tau''=(n,n-1,\ldots,r+1)$. We need to show that any subpermutation $abc$ of $\tau$ is not a $123$ pattern. Suppose, by way of contradiction, there exists a subpermutation $abc$ of $\tau$ which is a $123$ pattern. Then $a<b<c$. We split into two cases:

Case 1: $a\geq r$. Then $c>b>r$. It follows that $bc$ is an increasing subpermutation of $\tau''$. But this contradicts the fact that $\tau''=(n,n-1,\ldots,r+1)$.

Case 2: $a<r$. Since $\tau'$ avoids $123$, either $b>r$ or $c>r$. If $b>r$, then $c>b>r$. Using a similar argument as in Case 1, we have a contradiction. So we suppose $b<r$ and $c>r$. Here $b\neq r$ because $b\in D_\tau(a)$. Since $ab$ is an increasing subpermutation of $\tau'$, we have $b=\tau'(j)$ for some $j\geq i+1$. Hence $c\in \mathcal{A}_\tau(\tau'(i+1)) \subseteq \mathcal{A}_\tau(b)$ which is again a contradiction.

This proves that $\mathcal{P}\subseteq S_{n,r}(123)$, and, consequently, 
$S_{n,r}(123)=\mathcal{P}$. Therefore, by \cref{Lemma:AscendIn123Avoiding},  \[
|S_{n, r}(123)|=|\mathcal{P}|=\sum_{i=1}^{r-1}\binom{i+n-r}{i}a_{r-1}(i)=\sum_{i=1}^{r-1}\binom{i+n-r}{i}\left[\binom{2r-i-3}{r-2}-\binom{2r-i-3}{r-1}\right].\]
It remains to show that $\sum_{i=1}^{r-1}\binom{i+n-r}{i}\left[\binom{2r-i-3}{r-2}-\binom{2r-i-3}{r-1}\right]=b_{n,r}$. First, by the definition of binomial coefficients, we have
\[
\binom{2r-i-3}{r-2}-\binom{2r-i-3}{r-1}=\binom{2r-i-3}{r-2}-\frac{r-i-1}{r-1}\binom{2r-i-3}{r-2}=\frac{i}{r-1}\binom{2r-i-3}{r-2}.
\]
Hence
\[
\begin{split}
\sum_{i=1}^{r-1}\binom{i+n-r}{i}\left[\binom{2r-i-3}{r-2}-\binom{2r-i-3}{r-1}\right]=&\sum_{i=1}^{r-1}\binom{i+n-r}{i}\frac{i}{r-1}\binom{2r-i-3}{r-2}\\=&\sum_{i=1}^{r-1}\frac{n-r+1}{i}\binom{i+n-r}{i-1}\frac{i}{r-1}\binom{2r-i-3}{r-2}\\=&\frac{n-r+1}{r-1}\sum_{i=1}^{r-1}\binom{i+n-r}{i-1}\binom{2r-i-3}{r-2}.
\end{split}
\]
By a variation of Vandermonde's identity,
\[
\sum_{i=1}^{r-1} \binom{i+n-r}{i-1} \binom{2r-i-3}{r-2} = \binom{n+r-2}{r-2}.
\]
(One way to see this is by comparing the coefficients of $x^{r-2}$ on the two sides of the identity
$(1-x)^{-(n-r+2)} (1-x)^{-(r-1)} = (1-x)^{-(n+1)}$.)
Putting it all together,
\[
\begin{split}
\sum_{i=1}^{r-1}\binom{i+n-r}{i}\left[\binom{2r-i-3}{r-2}-\binom{2r-i-3}{r-1}\right]=&\frac{n-r+1}{r-1}\binom{n+r-2}{r-2}\\=&\frac{n-r+1}{n+r-1}\binom{n+r-1}{r-1}=b_{n,r}.
\end{split}
\]

This completes the proof of \cref{Theorem:ClassicalLeading123AND132}.


\section{Single patterns of length three}\label{Section:Single3}
In this section, we enumerate $S_{n,(c_1,c_2,\ldots,c_t)}(\sigma)$ for $\sigma\in S_3$. By \cref{lemma:complements}, it suffices to enumerate permutations avoiding the patterns $123$, $132$, and $231$. We start with the pattern $231$. The key features about $S_{n,(c_1,c_2,\ldots,c_t)}(231)$ are that the enumeration is related to the order statistics of $\{c_1,c_2,\ldots,c_t\}$ and there is a `block' structure for all $\tau\in S_{n,(c_1,c_2,\ldots,c_t)}(231)$.

\begin{theorem}\label{Theorem:LeadingTerms231}
    If $c_i<c_j$ for some $1\leq i<j\leq t$ and there exists $\alpha<c_i$ such that $\alpha\notin\{c_1,c_2,\ldots,c_t\}$, then $|S_{n,(c_1,c_2,\ldots,c_t)}(231)|=0$; otherwise, we have
\begin{equation}\label{Qformula}
    |S_{n,(c_1,c_2,\ldots,c_t)}(231)|=\prod_{k=1}^{t+1}C_{c_{(k)}-c_{(k-1)}-1},
\end{equation}
    where $C_i$ is the $i$th Catalan number, $c_{(0)}=0$, $c_{(t+1)}=n+1$, and $c_{(1)}<c_{(2)}<\cdots<c_{(t)}$ are the order statistics of $\{c_1,c_2,\ldots,c_t\}$.
\end{theorem}
\begin{proof}
    If $c_i<c_j$ for some $1\leq i<j\leq t$ and there exists $\alpha<c_i$ such that $\alpha\notin\{c_1,c_2,\ldots,c_t\}$, then $c_ic_j\alpha$ is a $231$ pattern. Therefore, $|S_{n,(c_1,c_2,\ldots,c_t)}(231)|=0$.

    Now suppose otherwise. We will build a set $\mathcal{Q}$ 
whose cardinality is given by the right hand side of (\ref{Qformula}) 
and then show that $\mathcal{Q}=S_{n,(c_1,c_2,\ldots,c_t)}(231)$. Let $\mathcal{Q}$ be the subset of $S_{n,(c_1,c_2,\ldots,c_t)}$ such that every $\tau\in\mathcal{Q}$ has the following properties:
\begin{itemize}
    \item[(i)] for all $k,\ell\in[t+1]$ with $k<\ell$, if \[x\in\left\{c_{(k-1)}+1,c_{(k-1)}+2,\ldots,c_{(k)}-1\right\}~\text{and}~y\in\left\{c_{(\ell-1)}+1,c_{(\ell-1)}+2,\ldots,c_{(\ell)}-1\right\},\] then $x\in\mathcal{A}_\tau(y)$, and
    \item[(ii)] for all $k\in[t+1]$, the subpermutation on $\left\{c_{(k-1)}+1,c_{(k-1)}+2,\ldots,c_{(k)}-1\right\}$ avoids $231$.
\end{itemize}
That is, for each $\tau\in\mathcal{Q}$, $(\tau(t+1),\tau(t+2),\ldots,\tau(n))$ is the concatenation of $t+1$ (some possibly empty) $231$-avoiding permutations, which we call $231$-avoiding blocks. 
For all $k\in[t+1]$, the $k$th block is a $231$-avoiding permutation on $\left\{c_{(k-1)}+1,c_{(k-1)}+2,\ldots,c_{(k)}-1\right\}$. 
Since for all $k\in[t+1]$,  \[\left|\left\{c_{(k-1)}+1,c_{(k-1)}+2,\ldots,c_{(k)}-1\right\}\right|=c_{(k)}-c_{(k-1)}-1,\] by \cref{Theorem:ClassicalSingle3}, the number of $231$-avoiding permutations on $\left\{c_{(k-1)}+1,c_{(k-1)}+2,\ldots,c_{(k)}-1\right\}$ is $C_{c_{(k)}-c_{(k-1)}-1}$. Hence we have $|\mathcal{Q}|=\prod_{k=1}^{t+1}C_{c_{(k)}-c_{(k-1)}-1}$.

Let $\tau\in S_{n,(c_1,c_2,\ldots,c_t)}(231)$. Then every subpermutation of $\tau$ avoids $231$. So, $\tau$ satisfies Property (ii). Now we show that $\tau$ satisfies Property (i). Suppose not. Then there exist $k,\ell\in[t]$ with $k<\ell$, $x\in\left\{c_{(k-1)}+1,c_{(k-1)}+2,\ldots,c_{(k)}-1\right\}$, and $y\in\left\{c_{(\ell-1)}+1,c_{(\ell-1)}+2,\ldots,c_{(\ell)}-1\right\}$, such that $x\in\mathcal{D}_\tau(y)$. It follows that $c_{(\ell-1)}yx$ is a $231$ 
pattern, which is a contradiction. Hence $\tau\in \mathcal{Q}$. This proves that $S_{n,(c_1,c_2,\ldots,c_t)}(231)\subseteq\mathcal{Q}$.

Now let $\tau\in\mathcal{Q}$. To show that $\tau\in S_{n,(c_1,c_2,\ldots,c_t)}(231)$, it suffices to show that any subpermutation $abc$ of $\tau$ is not a $231$ pattern. Suppose, by way of contradiction, that a subpermutation $abc$ of $\tau$ is a 231 pattern. Then we must have $c<a<b$ and $c\in D_\tau(b)$. We split into four cases:

Case 1: $a,b,c\notin\{c_1,c_2,\ldots,c_t\}$. Since $(\tau(t+1),\tau(t+2),\ldots,\tau(n))$ is a concatenation of $t+1$ (some possibly empty) $231$-avoiding blocks, $c$ must be in a block after the block $a$ is in. Since $c<a$, this contradicts Property (i).

Case 2: $a\in\{c_1,c_2,\ldots,c_t\}$ and $b,c\notin\{c_1,c_2,\ldots,c_t\}$. Then we have $a=c_{(k)}$ for some $k\in[t]$. Since $c<a<b$, $b$ and $c$ must be in two different blocks. Since $c\in D_\tau(b)$, $c$ is in a block after the block $b$ is in which contradicts Property (i).

Case 3: $a,b\in\{c_1,c_2,\ldots,c_t\}$ and $c\notin\{c_1,c_2,\ldots,c_t\}$. Then $a=c_i$ and $b=c_j$ with $1\leq i<j\leq t$, and $c<a$. This is a contradiction. 

Case 4: $a,b,c\in\{c_1,c_2,\ldots,c_t\}$. Then $abc$ is a subpermutation of $(c_1,c_2,\ldots,c_t)$ which contradicts our convention that $(c_1,c_2,\ldots,c_t)$ avoids $231$.

Hence, we have $\mathcal{Q}\subseteq S_{n,(c_1,c_2,\ldots,c_t)}(231)$.
\end{proof}

\begin{remark}
For all $n\geq3$, by \cref{Theorem:ClassicalSingle3} and \cref{Theorem:LeadingTerms231} with $t=1$, we have
\[
C_n=|S_n(231)|=\sum_{r=1}^n|S_{n,r}(231)|=\sum_{r=1}^nC_{n-r}C_{r-1}.
\]
This offers an alternative interpretation 
for the well-known recurrence relation for the 
Catalan numbers, see \cite[Section 3.2]{EgeciogluGarsia2021} and \cite[Section 1.2]{Stanley2015}. 
\end{remark}

In contrast to the $231$ pattern, the expressions for the $123$ and $132$ patterns are related to the minimum of $\{c_1,c_2,\ldots,c_t\}$. Recall that for a subset $A\subseteq[n]$, the \textit{standardization} of a permutation $\tau=(\tau(1),\tau(2),\ldots,\tau(|A|))$ is the permutation $s(\tau)\in S_{|A|}$ obtained by replacing the $i$th smallest entry in $\tau$ with $i$ for all $i$. We will need the following result:

\begin{lemma}\label{Lemma:StandardizationLeadingTerm}
Suppose $A\subseteq[n]$ such that there exists $r\in[n]$ with $[r]\subseteq A$, and $\sigma\in S_k$ with $k \leq |A|$. Then $s(S_{A,r}(\sigma))=S_{|A|,r}(\sigma)$, where $s(S_{A,r}(\sigma))=\{s(\tau):\tau\in S_{A,r}(\sigma)\}$.
\end{lemma}

\begin{proof}
Let $\tau\in S_{A,r}(\sigma)$. Since $[r]\subseteq A$, $\tau(1)$ is the $r$th smallest number in $\tau$. So $s(\tau)\in S_{|A|,r}$. By \cref{Lemma:StandardizationPreservesPatternAvoidance}, we have $s(S_{A,r}(\sigma))\subseteq S_{|A|,r}(\sigma)$.

Now let $\tau\in S_{|A|,r}(\sigma)$. Let $a_1<a_2<\cdots<a_{|A|}$ be the elements in $A$. Since $[r]\subseteq A$, we have $a_r=r$. Let $\tau'$ be the matching permutation of $\tau$ on $A$. Since $\tau(1)=r$, we have $\tau'(1)=a_r=r$. It is not hard to see that $\tau'$ avoids $\sigma$. Hence $\tau'\in S_{A,r}(\sigma)$. By our construction, we also have $s(\tau')=\tau$. So $\tau\in s(S_{A,r}(\sigma))$. Hence, $S_{|A|,r}(\sigma)\subseteq s(S_{A,r}(\sigma))$.
\end{proof}
\begin{theorem}\label{Theorem:123LeadingTerms}
    If $c_i<c_j$ for some $1\leq i<j\leq t$ and there exists $\alpha>c_j$ such that $\alpha\notin\{c_1,c_2,\ldots,c_t\}$, then $|S_{n,(c_1,c_2,\ldots,c_t)}(123)|=0$; otherwise, we have
    \[
    |S_{n,(c_1,c_2,\ldots,c_t)}(123)|=|S_{n-t+1,\min\{c_1,c_2,\ldots,c_t\}}(123)|=b_{n-t+1,\min\{c_1,c_2,\ldots,c_t\}}.
    \]
\end{theorem}
\begin{proof}
    If $c_i<c_j$ for some $i<j$ and there exists $\alpha>c_j$ such that $\alpha\notin\{c_1,c_2,\ldots,c_t\}$, then $c_ic_j\alpha$ is a $123$ pattern. Therefore, $|S_{n,(c_1,c_2,\ldots,c_t)}(123)|=0$.

    Now suppose otherwise. For simplicity, we write $x=\min\{c_1,c_2,\ldots,c_t\}$ and \[A=([n]\backslash\{c_1,c_2,\ldots,c_t\})\cup\{x\}.\] Consider the map $f:S_{n,(c_1,c_2,\ldots,c_t)}(123)\to S_{A,x}(123)$ such that for all $\tau\in S_{n,(c_1,c_2,\ldots,c_t)}(123)$,
    \[
    f(\tau)=(x,\tau(t+1),\tau(t+2),\ldots,\tau(n)).
    \]
    This is well defined since $(x,\tau(t+1),\tau(t+2),\ldots,\tau(n))$ is a subpermutation of $\tau$. If $\tau,\tau'\in S_{n,(c_1,c_2,\ldots,c_t)}(123)$ with $\tau\neq\tau'$, then $\tau(i)\neq \tau'(i)$ for some $i\in\{t+1,t+2,\ldots,n\}$ and hence $(x,\tau(t+1),\tau(t+2),\ldots,\tau(n))\neq (x,\tau'(t+1),\tau'(t+2),\ldots,\tau'(n))$. So $f$ is injective. $f$ is also surjective because for all $\tau=(x,\tau(2),\tau(3),\ldots,\tau(|A|))\in S_{A,x}(123)$, we have \[\tau'=(c_1,c_2,\ldots,c_t,\tau(2),\tau(3),\ldots,\tau(|A|))\in S_{n,(c_1,c_2,\ldots,c_t)}(123)\] and $f(\tau')=\tau$. Therefore, $f$ is a bijection and hence $|S_{n,(c_1,c_2,\ldots,c_t)}(123)|=|S_{A,x}(123)|$.

    Since $x=\min\{c_1,c_2,\ldots,c_t\}$, we have $[x]\subseteq A$. Now by \cref{Lemma:StandardizationLeadingTerm}, we have
    \[
    |S_{n,(c_1,c_2,\ldots,c_t)}(123)|=|S_{A,x}(123)|=|s(S_{A,x}(123))|=|S_{|A|,x}(123)|=|S_{n-t+1,x}(123)|.\qedhere
    \]
    \end{proof}

The result for the $132$ pattern is similar to the $123$ pattern and hence we state the following result without proof.

\begin{theorem}\label{Theorem:132LeadingTerms}
If $c_i<c_j$ for some $i<j$ and there exists $\alpha$ such that $c_i<\alpha<c_j$ and $\alpha\notin\{c_1,c_2,\ldots,c_t\}$, then $|S_{n,(c_1,c_2,\ldots,c_t)}(132)|=0$. Otherwise, we have
    \[|S_{n,(c_1,c_2,\ldots,c_t)}(132)|=|S_{n-t+1,\min\{c_1,c_2,\ldots,c_t\}}(132)|=b_{n-t+1,\min\{c_1,c_2,\ldots,c_t\}}.
    \]
\end{theorem}


\section{Pairs of patterns of length three}\label{Section:Pair3}
In this section, we enumerate permutations with fixed prefix $(c_1,c_2,\ldots,c_t)$ 
which 
avoid a pair of patterns $\{\sigma_1,\sigma_2\}$ of length three. Recall that we use $S_n(\sigma_1,\sigma_2)$ to denote the set of permutations $\tau\in S_n$ such that $\tau$ avoids both $\sigma_1$ and $\sigma_2$. 
We need the following results by Simion and Schmidt \cite[Section 3]{SimionSchmidt1985}:

\begin{theorem}{\rm\cite[Section 3]{SimionSchmidt1985}}\label{Theorem:Pairs3}
For all $n\geq1$,
    \[
   \begin{split}
|S_n(123,132)|=&|S_n(321,312)|=|S_n(123,213)|=|S_n(321,231)|=|S_n(132,213)|=|S_n(312,231)|\\=&|S_n(132,231)|=|S_n(312,213)|=|S_n(132,312)|=|S_n(213,231)|=2^{n-1},
   \end{split} 
    \]
    \[
    |S_n(123,312)|=|S_n(321,132)|=|S_n(123,231)|=|S_n(321,213)|=\binom{n}{2}+1,
    \]
    and
\[
|S_n(123,321)|=\left\{\begin{array}{ll}
 0    & \text{ if }~~n\geq5, \\
 n    & \text{ if }~~n=1\text{ or }n=2,\\
 4    & \text{ if }~~n=3\text{ or }n=4.
\end{array}\right.
\]
\end{theorem}

Out of the 15 pairs of patterns of length 3, there are three self-complementary pairs: $\{123,321\}$, $\{132,312\}$, and $\{213,231\}$. That is, $\{123^c,321^c\}=\{123,321\}$, $\{132^c,312^c\}=\{132,312\}$, and $\{213^c,231^c\}=\{213,231\}$. By the Erd\H{o}s-Szekeres theorem \cite[p. 467]{ErdosSzekeres1935}, 
for $n\geq5$, every $\tau\in S_n$ has either an increasing 
or a decreasing subpermutation of length three. 
Hence $|S_{n,(c_1,c_2,\ldots,c_t)}(123,321)|=0$ if $n\geq5$. Since one could routinely calculate $|S_{n,(c_1,c_2,\ldots,c_t)}(123,321)|$ when $n\leq4$, we do not include the exact results for the pair $\{123, 321\}$ here. We start with $\{132,312\}$ and $\{213,231\}$.

\begin{theorem}
If $\{c_1,c_2,\ldots,c_t\}$ is a set of consecutive integers, then 
\[
|S_{n,(c_1,c_2,\ldots,c_t)}(132,312)|=\binom{n-t}{\min\{c_1,c_2,\ldots,c_t\}-1};
\]
otherwise, $|S_{n,(c_1,c_2,\ldots,c_t)}(132,312)|=0$.
\end{theorem}
\begin{proof}
Write $x=\min\{c_1,c_2,\ldots,c_t\}$ and $A=[n]\backslash\{c_1,c_2,\ldots,c_t\}$.

Suppose that $\{c_1,c_2,\ldots,c_t\}$ is a set of consecutive integers. Then we have \[\{c_1,c_2,\ldots,c_t\}=\{x,x+1,\ldots,x+t-1\}.\] Let $\tau\in S_{n,(c_1,c_2,\ldots,c_t)}(132,312)$. Since $\tau$ avoids $132$, the subpermutation on $\{x+t,x+t+1,\ldots,n\}$, if $x+t-1\neq n$, is $(x+t,x+t+1,\ldots,n)$; and since $\tau$ avoids $312$, the subpermutation on $\{1,2,\ldots,x-1\}$, if $x\neq 1$, is $(x-1,x-2,\ldots,1)$. The number of shuffles of $(x+t,x+t+1,\ldots,n)$ and $(x-1,x-2,\ldots,1)$ is $\binom{n-t}{x-1}$. Hence $|S_{n,(c_1,c_2,\ldots,c_t)}(132,312)|\leq \binom{n-t}{x-1}$. Now let $\tau\in S_{n,(c_1,c_2,\ldots,c_t)}$ such that $(\tau(t+1),\tau(t+2),\ldots,\tau(n))$ is a shuffle of $(x+t,x+t+1,\ldots,n)$ and $(x-1,x-2,\ldots,1)$. It is easy to check that $\tau\in S_{n,(c_1,c_2,\ldots,c_t)}(132,312)$. Hence we have $|S_{n,(c_1,c_2,\ldots,c_t)}(132,312)|=\binom{n-t}{x-1}$.

Now suppose that $\{c_1,c_2,\ldots,c_t\}$ is not a set of consecutive integers. Then there exists $y\in A$ and $i,j\in[t]$ such that $i<j$, and $c_i<y<c_j$ or $c_j<y<c_i$. Then $c_ic_jy$ is either a $132$ pattern or a $312$ pattern. Hence, we have $|S_{n,(c_1,c_2,\ldots,c_t)}(132,312)|=0$.
\end{proof}

\begin{theorem}\label{Theorem:LeadingTerms213;231}
If $(c_1,c_2,\ldots,c_t)$ is a shuffle of $(1,2,\ldots,t-s)$ and $(n,n-1,\ldots,n-s+1)$ for some $s\in\{0,1,\ldots,t\}$, then
\[|S_{n,(c_1,c_2,\ldots,c_t)}(213,231)|=2^{n-t-1};\]
otherwise, $|S_{n,(c_1,c_2,\ldots,c_t)}(213,231)|=0$.
\end{theorem}
Note that, in \cref{Theorem:LeadingTerms213;231}, when $s=0$, we mean that $(c_1,c_2,\ldots,c_t)=(1,2,\ldots,t)$; and when $s=t$, we mean that $(c_1,c_2,\ldots,c_t)=(n,n-1,\ldots,n-t+1)$.
\begin{proof}
Write $A=[n]\backslash\{c_1,c_2,\ldots,c_t\}$.

Suppose $(c_1,c_2,\ldots,c_t)$ is a shuffle of $(1,2,\ldots,t-s)$ and $(n,n-1,\ldots,n-s+1)$ for some $s\in\{0,1,\ldots,t\}$. Consider the map $f:S_{n,(c_1,c_2,\ldots,c_t)}(213,231)\to S_{n-t}(213,231)$ such that for all $\tau\in S_{n,(c_1,c_2,\ldots,c_t)}(213,231)$,
\[
f(\tau)=s(\tau(t+1),\tau(t+2),\ldots,\tau(n))
\]
where $s(x)$ is the standardization of $x$. We will show that $f$ is a bijection. It is easy to see that $f$ is a well-defined function. Similarly to the logic in the proof of \cref{Theorem:123LeadingTerms}, one can also see that $f$ is injective. It remains to show that $f$ is surjective.

Let $\tau\in S_{n-t}(213,231)$ and let $\tau'$ be the matching permutation of $\tau$ on $A$. Then $\pi:=(c_1,c_2,\ldots,c_t,\tau'(1),\tau'(2),\ldots,\tau'(n-t))\in S_{n,(c_1,c_2,\ldots,c_t)}$. We need to show that $\pi$ avoids $213$ and $231$. Let $xyz$ be a subpermutation of $\pi$. If $xyz$ is a subpermutation of $(c_1,c_2,\ldots,c_t)$ or $\tau'$, then $xyz$ is neither a $213$ pattern nor a $231$ pattern. So we assume that $x\in\{c_1,c_2,\ldots,c_t\}$ and $z\in A$. We split into two cases:

Case 1: $x>z$. Then $xyz$ is not a $213$ pattern. Since $(c_1,c_2,\ldots,c_t)$ is a shuffle of $(1,2,\ldots,t-s)$ and $(n,n-1,\ldots,n-s+1)$, we have $x\geq n-s+1$. Since $y\in\mathcal{D}_\pi(x)$, we have $y<x$ and hence $xyz$ is not a $231$ pattern.

Case 2: $x<z$. Then $xyz$ is not a $231$ pattern. Since $(c_1,c_2,\ldots,c_t)$ is a shuffle of $(1,2,\ldots,t-s)$ and $(n,n-1,\ldots,n-s+1)$, we have $x\leq t-s$. Since $y\in\mathcal{D}_\pi(x)$, we have $y>x$ and hence $xyz$ is not a $213$ pattern.

Hence $\pi\in S_{n,(c_1,c_2,\ldots,c_t)}(213,231)$ and $f(\pi)=\tau$ by our construction. This shows that $f$ is surjective. Now $f$ is a bijection and hence, by \cref{Theorem:Pairs3},
\[
|S_{n,(c_1,c_2,\ldots,c_t)}(213,231)|=|S_{n-t}(213,231)|=2^{n-t-1}.
\]

Now suppose $(c_1,c_2,\ldots,c_t)$ is not a shuffle of $(1,2,\ldots,t-s)$ and $(n,n-1,\ldots,n-s+1)$ for any $s\in\{0,1,\ldots,t\}$. There are two scenarios where this could happen. The first one is when $\{c_1,c_2,\ldots,c_t\}\neq \{1,2,\ldots,t-s,n-s+1,n-s+2,\ldots,n\}$, and the second one is when $\{c_1,c_2,\ldots,c_t\}=\{1,2,\ldots,t-s,n-s+1,n-s+2,\ldots,n\}$ but for any $s\in\{0,1,\ldots,t\}$, either the subpermutation of $(c_1,c_2,\ldots,c_t)$ on $\{1,2,\ldots,t-s\}$ is not $(1,2,\ldots,t-s)$ or the subpermutation of $(c_1,c_2,\ldots,c_t)$ on $\{n-s+1,n-s+2,\ldots,n\}$ is not $(n,n-1,\ldots,n-s+1)$. We split into two cases based on these scenarios. Let $\tau\in S_{n,(c_1,c_2,\ldots,c_t)}$.

Case 3: $\{c_1,c_2,\ldots,c_t\}\neq \{1,2,\ldots,t-s,n-s+1,n-s+2,\ldots,n\}$ for any $s\in\{0,1,\ldots,t\}$. Then there exist $x\in\{c_1,c_2,\ldots,c_t\}$ and $y,z\in A$ such that $y>x$ and $z<x$. In this case, either $xyz$ or $xzy$ is a subpermutation of $\tau$ and hence $\tau$ contains either a $213$ pattern or a $231$ pattern. Hence, $|S_{n,(c_1,c_2,\ldots,c_t)}(213,231)|=0$. 

Case 4: $\{c_1,c_2,\ldots,c_t\}=\{1,2,\ldots,t-s,n-s+1,n-s+2,\ldots,n\}$ for some $s\in\{0,1,\ldots,t\}$. 

Subcase 4.1: The subpermutation of $(c_1,c_2,\ldots,c_t)$ on $\{1,2,\ldots,t-s\}$ is not $(1,2,\ldots,t-s)$. Then there exist $x,y\in\{c_1,c_2,\ldots,c_t\}$ and $z\in A$ such that $x<y<z$ and $y\in\mathcal{A}_\tau(x)$. Now $yxz$ is a $213$ pattern and hence $|S_{n,(c_1,c_2,\ldots,c_t)}(213,231)|=0$.

Subcase 4.2: The subpermutation of $(c_1,c_2,\ldots,c_t)$ on $\{n-s+1,n-s+2,\ldots,n\}$ is not $(n,n-1,\ldots,n-s+1)$. Then there exist $x,y\in\{c_1,c_2,\ldots,c_t\}$ and $z\in A$ such that $z<y<x$ and $y\in\mathcal{A}_\tau (x)$. Now $yxz$ is a $231$ pattern and hence $|S_{n,(c_1,c_2,\ldots,c_t)}(213,231)|=0$.
\end{proof}

We have 12 pairs left to consider. By \cref{lemma:complements}, it suffices to look at $\{123,132\}$, $\{123,213\}$, $\{132,213\}$, $\{132,231\}$, $\{123,312\}$, and $\{123,231\}$.

\begin{theorem}\label{Theorem:LeadingTerms123;132}
Write $\alpha=\max([n]\backslash\{c_1,c_2,\ldots,c_t\})$. If $\{c_1,c_2,\ldots,c_{n-\alpha}\}\neq \{\alpha+1,\alpha+2,\ldots,n\}$ or $(c_{n-\alpha+1},c_{n-\alpha+2},\ldots,c_t)\neq (\alpha-1,\alpha-2,\ldots,n-t)$, then $|S_{n,(c_1,c_2,\ldots,c_t)}(123,132)|=0$; otherwise,
\[|S_{n,(c_1,c_2,\ldots,c_t)}(123,132)|=2^{n-t-1}.\]
\end{theorem}
Note that, in \cref{Theorem:LeadingTerms123;132}, we must have $\alpha\geq n-t$, and if $\alpha=n-t$, then $\{c_1,c_2,\ldots,c_t\}=\{n-t+1,n-t+2,\ldots,n\}$.
\begin{proof}
    Write $A=[n]\backslash\{c_1,c_2,\ldots,c_t\}$.

    First suppose that $\{c_1,c_2,\ldots,c_{n-\alpha}\}\neq \{\alpha+1,\alpha+2,\ldots,n\}$. As mentioned earlier, if $\alpha=n-t$, then $\{c_1,c_2,\ldots,c_{n-\alpha}\}=\{c_1,c_2,\ldots,c_t\}=\{n-t+1,n-t+2,\ldots,n\}$. So we must have $\alpha>n-t$. Since $\alpha=\max A$, there exist $i\in\{1,2,\ldots,n-\alpha\}$ and $j\in\{n-\alpha+1,n-\alpha+2,\ldots,t\}$ such that $c_i<\alpha<c_j$. So $c_ic_j\alpha$ is a $132$ pattern and hence $|S_{n,(c_1,c_2,\ldots,c_t)}(123,132)|=0$.

    Next suppose that $\{c_1,c_2,\ldots,c_{n-\alpha}\}=\{\alpha+1,\alpha+2,\ldots,n\}$ but $(c_{n-\alpha+1},c_{n-\alpha+2},\ldots,c_t)\neq (\alpha-1,\alpha-2,\ldots,n-t)$. Notice that this could only happen when $\alpha>n-t$ because otherwise $n-\alpha+1=n-(n-t)+1=t+1$. We split into three cases:

    Case 1: $\{c_{n-\alpha+1},c_{n-\alpha+2},\ldots,c_t\}\neq\{\alpha-1,\alpha-2,\ldots,n-t\}$ and there exists $i\in\{n-\alpha+1,n-\alpha+2,\ldots,t\}$ such that $c_i>\alpha$. Then $\{c_1,c_2,\ldots,c_{n-\alpha}\}\neq \{\alpha+1,\alpha+2,\ldots,n\}$. This is a contradiction.

    Case 2: $\{c_{n-\alpha+1},c_{n-\alpha+2},\ldots,c_t\}\neq\{\alpha-1,\alpha-2,\ldots,n-t\}$ and $c_i<\alpha$ for all $i\in\{n-\alpha+1,n-\alpha+2,\ldots,t\}$. Let $\tau\in S_{n,(c_1,c_2,\ldots,c_t)}$. Then there exist $y\in A$ and $c_i\in \{c_{n-\alpha+1},c_{n-\alpha+2},\ldots,c_t\}$ such that $c_i<y<\alpha$. So either $c_iy\alpha$ or $c_i\alpha y$ is a subpermutation of $\tau$. It follows that $\tau$ has either a $123$ pattern or a $132$ pattern. Hence, $|S_{n,(c_1,c_2,\ldots,c_t)}(123,132)|=0$.

    Case 3: $\{c_{n-\alpha+1},c_{n-\alpha+2},\ldots,c_t\}=\{\alpha-1,\alpha-2,\ldots,n-t\}$ but $(c_{n-\alpha+1},c_{n-\alpha+2},\ldots,c_t)\neq (\alpha-1,\alpha-2,\ldots,x)$. Then there exist $i,j\in\{n-\alpha+1,n-\alpha+2,\ldots,t\}$ with $i<j$ and $c_i<c_j$. Now $c_ic_j\alpha$ is a $123$ pattern. Hence, $|S_{n,(c_1,c_2,\ldots,c_t)}(123,132)|=0$.



    The proof that $|S_{n,(c_1,c_2,\ldots,c_t)}(123,132)|=2^{n-t-1}$ when $\{c_1,c_2,\ldots,c_{n-\alpha}\}=\{\alpha+1,\alpha+2,\ldots,n\}$ and $(c_{n-\alpha+1},c_{n-\alpha+2},\ldots,c_t)=(\alpha-1,\alpha-2,\ldots,n-t-1)$ is similar to the proof of \cref{Theorem:LeadingTerms213;231}.
\end{proof}

\begin{theorem}\label{Theorem:LeadingTerms123AND213}
If there exists $\alpha\in[n]\backslash\{c_1,c_2,\ldots,c_t\}$ and $1\leq i<j\leq t$ with $c_i,c_j<\alpha$, then $|S_{n,(c_1,c_2,\ldots,c_t)}(123,213)|=0$; otherwise,
\[
|S_{n,(c_1,c_2,\ldots,c_t)}(123,213)|=2^{\max\{0,\min\{c_1,c_2,\ldots,c_t\}-2\}}.
\]
\end{theorem}
\begin{proof}
Write $x=\min\{c_1,c_2,\ldots,c_t\}$ and $A=[n]\backslash\{c_1,c_2,\ldots,c_t\}$.

    Suppose there exist $\alpha\in A$ and $1\leq i<j\leq t$ with $c_i,c_j<\alpha$. Then $c_ic_j\alpha$ is either a $123$ pattern or a $213$ pattern. Hence $|S_{n,(c_1,c_2,\ldots,c_t)}(123,213)|=0$.

    Suppose there do not exist $\alpha\in A$ and $1\leq i<j\leq t$ with $c_i,c_j<\alpha$. Let $\tau\in S_{n,(c_1,c_2,\ldots,c_t)}(123,213)$.
    
    Case 1: $x=1$. Then $(\tau(t+1),\tau(t+2),\ldots,\tau(n))$ is a decreasing sequence. Otherwise, there would exist $y,z\in A$ such that $1yz$ is a $123$ pattern. Hence $|S_{n,(c_1,c_2,\ldots,c_t)}(123,213)|=1$.

    Case 2: $x\geq2$.  Since $\tau$ avoids $123$, the subpermutation of $\tau$ on $A\cap\{x+1,x+2,\ldots,n\}$ is decreasing. Since $\tau$ avoids $213$, for all $y\in\{1,2,\ldots,x-1\}$ and for all $z\in A\cap\{x+1,x+2,\ldots,n\}$, $z\in \mathcal{A}_\pi(y)$. So by \cref{Theorem:Pairs3}, we have \[|S_{n,(c_1,c_2,\ldots,c_t)}(123,213)|\leq |S_{x-1}(123,213)|=2^{x-2}.\] It is easy to check that for all $\tau'\in S_{x-1}(123,213)$ and decreasing subpermutation $\tau''$ on $A\backslash[x-1]$, we have \[(c_1,c_2,\ldots,c_t,\tau''(1),\tau''(2),\ldots,\tau''(n-t-x+1),\tau'(1),\tau'(2),\ldots,\tau'(x-1))\in S_{n,(c_1,c_2,\ldots,c_t)}(123,213).\]
Hence, we have \[|S_{n,(c_1,c_2,\ldots,c_t)}(123,213)|=|S_{x-1}(123,213)|=2^{x-2}.\qedhere\]
\end{proof}

\begin{theorem}
If there exist $\alpha,\beta\in[n]\backslash\{c_1,c_2,\ldots,c_t\}$ and $i\in[t]$ with $\min\{c_1,c_2,\ldots,c_t\}<\alpha<c_i<\beta$, then $|S_{n,(c_1,c_2,\ldots,c_t)}(132,213)|=0$; if there exist $\alpha\in[n]\backslash\{c_1,c_2,\ldots,c_t\}$ and $i,j\in[t]$ with $i<j$ and $c_i<\alpha<c_j$ or $c_j<c_i<\alpha$, then $|S_{n,(c_1,c_2,\ldots,c_t)}(132,213)|=0$; otherwise,
\[
    |S_{n,(c_1,c_2,\ldots,c_t)}(132,213)|=2^{\max\{0,\min\{c_1,c_2,\ldots,c_t\}-2\}}.
    \]
\end{theorem}
\begin{proof}
    Write $x=\min\{c_1,c_2,\ldots,c_t\}$ and $A=[n]\backslash\{c_1,c_2,\ldots,c_t\}$.

    Suppose there exist $\alpha,\beta\in A$ and $i\in[t]$ with $x<\alpha<c_i<\beta$. Let $\tau\in S_{n,(c_1,c_2,\ldots,c_t)}$. If $\alpha\in\mathcal{A}_\tau(\beta)$, then $c_i\alpha\beta$ is a $213$ pattern; and if $\alpha\in\mathcal{D}_\tau(\beta)$, then $x\beta\alpha$ is a $132$ pattern. Hence, $|S_{n,(c_1,c_2,\ldots,c_t)}(132,213)|=0$.

    Now suppose there exist $\alpha\in A$ and $i,j\in[t]$ with $i<j$ and $c_i<\alpha<c_j$ or $c_j<c_i<\alpha$. If $c_i<\alpha<c_j$, then $c_ic_j\alpha$ is a $132$ pattern; and if $c_j<c_i<\alpha$, then $c_ic_j\alpha$ is a $213$ pattern. Hence $|S_{n,(c_1,c_2,\ldots,c_t)}(132,213)|=0$.

    The rest of the proof is similar to the proof of \cref{Theorem:LeadingTerms123AND213}.
\end{proof}

\begin{theorem}
If there exist $\alpha\in[n]\backslash\{c_1,c_2,\ldots,c_t\}$ and $i,j\in[t]$ with $i<j$ and $c_i<\alpha<c_j$ or $\alpha<c_i<c_j$, then $|S_{n,(c_1,c_2,\ldots,c_t)}(132,231)|=0$; otherwise,
\[
    |S_{n,(c_1,c_2,\ldots,c_t)}(132,231)|=2^{\max\{0,\min\{c_1,c_2,\ldots,c_t\}-2\}}.
    \]
\end{theorem}
\begin{proof}
    Write $x=\min\{c_1,c_2,\ldots,c_t\}$ and $A=[n]\backslash\{c_1,c_2,\ldots,c_t\}$.

    Suppose there exist $\alpha\in A$ and $i,j\in[t]$ with $i<j$ and $c_i<\alpha<c_j$ or $\alpha<c_i<c_j$. If $c_i<\alpha<c_j$, then $c_ic_j\alpha$ is a $132$ pattern; and if $\alpha<c_i<c_j$, then $c_ic_j\alpha$ is a $231$ pattern. Hence, $|S_{n,(c_1,c_2,\ldots,c_t)}(132,231)|=0$.

    The rest of the proof is similar to the proof of \cref{Theorem:LeadingTerms123AND213}.
\end{proof}

\begin{theorem}
If $\{c_1,c_2,\ldots,c_t\}$ is a set of consecutive integers, then \[|S_{n,(c_1,c_2,\ldots,c_t)}(123,312)|=0,~1,\ \text{or}\ \min\{c_1,c_2,\ldots,c_t\};\] otherwise, $|S_{n,(c_1,c_2,\ldots,c_t)}(123,312)|=0$ or $1$.
\end{theorem}
\begin{proof}
Write $x=\min\{c_1,c_2,\ldots,c_t\}$, $y=\max\{c_1,c_2,\ldots,c_t\}$, and $A=[n]\backslash\{c_1,c_2,\ldots,c_t\}$. 

Suppose $\{c_1,c_2,\ldots,c_t\}$ is a set of consecutive integers. Then $\{c_1,c_2,\ldots,c_t\}=\{x,x+1,\ldots,x+t-1\}$. We split into three cases:

Case 1: There exist $\alpha\in A$ and $1\leq i<j\leq t$ with $c_i<c_j<\alpha$. Then $c_ic_j\alpha$ is a $123$ pattern and hence $|S_{n,(c_1,c_2,\ldots,c_t)}(123,312)|=0$.

Case 2: $x=n-t+1$ and there do not exist $\alpha\in A$ and $1\leq i<j\leq t$ with $c_i<c_j<\alpha$. Let $\tau\in S_{n,(c_1,c_2,\ldots,c_t)}(123,312)$. Then $(\tau(t+1),\tau(t+2),\ldots,\tau(n))$ is a subpermutation on $\{1,2,\ldots,x-1\}$. If $(\tau(t+1),\tau(t+2),\ldots,\tau(n))$ is not decreasing, then we would have a $312$ pattern. It is easy to check that $(c_1,c_2,\ldots,c_t,x-1,x-2,\ldots,1)$ avoids both $123$ and $312$ patterns. Hence $|S_{n,(c_1,c_2,\ldots,c_t)}(123,312)|=1$.

Case 3: $x<n-t+1$ and there do not exist $\alpha\in A$ and $1\leq i<j\leq t$ with $c_i<c_j<\alpha$. In this case, we have $y<n$. Let $\tau\in S_{n,(c_1,c_2,\ldots,c_t)}(123,312)$. Let $\tau'$ be the subpermutation on $\{1,2,\ldots,x-1\}$ and let $\tau''$ be the subpermutation on $\{y+1,y+2,\ldots,n\}$. 
        Since $\tau$ avoids $123$, $\tau''=(n,n-1,\ldots,y+1)$, and since $\tau$ avoids $312$, $\tau'=(x-1,x-2,\ldots,1)$. Moreover, if $\tau(i)\in\{1,2,..,x-1\}$ and $\tau(j),\tau(k)\in\{y+1,y+2,\ldots,n\}$ with $j<k$, 
either $\tau(j),\tau(k)\in\mathcal{A}_\tau(\tau(i))$ or $\tau(j),\tau(k)\in\mathcal{D}_\tau(\tau(i))$. Otherwise, $\tau(j)\tau(i)\tau(k)$ would be a $312$ pattern. Therefore
the number of shuffles of $\tau'$ and $\tau''$ that do not create a $312$ pattern is 
simply $\binom{x-1+1}{1}=x$. It is easy to check that none of these shuffles creates a $123$ pattern. Hence $|S_{n,(c_1,c_2,\ldots,c_t)}(123,312)|=x$.

Now suppose $\{c_1,c_2,\ldots,c_t\}$ is not a set of consecutive integers. There are three cases:

Case 4: If there exists $\alpha\in A$, $1\leq i<j\leq t$ with $c_i<c_j<\alpha$, then $c_ic_j\alpha$ is a $123$ pattern and hence $|S_{n,(c_1,c_2,\ldots,c_t)}(123,312)|=0$.

Case 5: If there exists $\alpha\in A$, $1\leq i<j\leq t$ with $c_j<\alpha<c_i$, then $c_ic_j\alpha$ is a $312$ pattern and hence $|S_{n,(c_1,c_2,\ldots,c_t)}(123,312)|=0$.

Case 6: The conditions in Case 4 and Case 5 are not met. Let $\tau\in S_{n,(c_1,c_2,\ldots,c_t)}(123,312)$. Let $\tau'$ be the subpermutation on $\{1,2,\ldots,y-1\}\cap A$ and let $\tau''$ be the subpermutation on $\{x+1,x+2,\ldots,n\}\cap A$. 
        Since $\tau$ avoids both $123$ and $312$, both $\tau'$ and $\tau''$ are decreasing. Since $\{c_1,c_2,\ldots,c_t\}$ is not a set of consecutive integers, there exists $\alpha\in A$ with $x<\alpha<y$. Since $\alpha$ is in both $\tau'$ and $\tau''$, $(\tau(t+1),\tau(t+2),\ldots,\tau(n))$ is a decreasing permutation on $A$. It is easy to see that if the conditions in Case 4 and Case 5 are not met and $(\tau(t+1),\tau(t+2),\ldots,\tau(n))$ is decreasing, then $\tau$ avoids both $123$ and $312$ patterns. Hence we have $|S_{n,(c_1,c_2,\ldots,c_t)}(123,312)|=1$.
\end{proof}
\begin{theorem}
    If $(c_1,c_2,\ldots,c_t)=(n,n-1,\ldots,n-t+1)$, then 
    \[
    |S_{n,(c_1,c_2,\ldots,c_t)}(123,231)|=\binom{n-t}{2}+1;
    \]
    otherwise, $|S_{n,(c_1,c_2,\ldots,c_t)}(123,231)|=0$ or $1$.
\end{theorem}
\begin{proof}
Write $x=\min\{c_1,c_2,\ldots,c_t\}$ and $A=[n]\backslash\{c_1,c_2,\ldots,c_t\}$.

    Suppose $(c_1,c_2,\ldots,c_t)=(n,n-1,\ldots,n-t+1)$. It is easy to see that $\tau\in S_{n,(c_1,c_2,\ldots,c_t)}$ avoids both $123$ and $231$ if and only if $(\tau(t+1),\tau(t+2),\ldots,\tau(n))$ avoids both $123$ and $231$. So by \cref{Theorem:Pairs3},
\[
|S_{n,(c_1,c_2,\ldots,c_t)}(123,231)|=|S_{n-t}(123,231)|=\binom{n-t}{2}+1.
\]

Now suppose $(c_1,c_2,\ldots,c_t)\neq (n,n-1,\ldots,n-t+1)$. We split into three cases:

Case 1: $\{c_1,c_2,\ldots,c_t\}=\{n-t+1,n-t+2,\ldots,n\}$. Then there exist indices $1\leq i<j\leq t$ and $\alpha\in A$ such that $\alpha<c_i<c_j$. So $c_ic_j\alpha$ is a $231$ pattern. Hence $|S_{n,(c_1,c_2,\ldots,c_t)}(123,231)|=0$.

Case 2: $\{c_1,c_2,\ldots,c_t\}\neq \{n-t+1,n-t+2,\ldots,n\}$ and there exist indices $1\leq i<j\leq t$ and $\alpha\in A$ such that $c_i<c_j<\alpha$ or $\alpha<c_i<c_j$, then $c_ic_j\alpha$ is either a $123$ pattern or a $231$ pattern. Hence $|S_{n,(c_1,c_2,\ldots,c_t)}(123,231)|=0$.

Case 3: The conditions for Case 1 and Case 2 are not met. Let $\tau\in S_{n,(c_1,c_2,\ldots,c_t)}(123,231)$. Since $\tau$ avoids $123$, the subpermutation $\tau'$ of $\tau$ on $\{x+1,x+2,\ldots,n\}\cap A$ is decreasing. Since $\tau$ avoids $231$, $\{x+1,x+2,\ldots,n\}\cap A\subseteq\mathcal{D}_\tau(i)$ for all $i<x$. Notice that since $\{c_1,c_2,\ldots,c_t\}\neq \{n-t+1,n-t+2,\ldots,n\}$, there exists $\alpha\in A$ with $\alpha>x$. Now if the subpermutation $\tau''$ of $\tau$ on $\{1,2,\ldots,x-1\}$ is not decreasing, we would have a $123$ pattern. So \[\tau=(c_1,c_2,\ldots,c_t,x-1,x-2,\ldots,1, \tau'(1),\tau'(2),\ldots,\tau'(n-t-x+1)).\] It is easy to check that $\tau$ avoids both $123$ and $231$. Hence, we have $|S_{n,(c_1,c_2,\ldots,c_t)}(123,231)|=1$.
\end{proof}


\section{The pair 3412 and 3421}\label{Section:Pair4}

The goal of this section is to show that the counting argument used in the 
proof of \cref{Theorem:LeadingTerms231} can be generalized to permutations avoiding both 3412 and 3421. We need the following result proved by Kremer \cite[Corollary 9]{Kremer2000}:
\begin{theorem}\label{Theorem:34123421}

For all $n\geq1$,
\[
|S_n(3412,3421)|=\mathbb{S}_{n-1},
\]
where $\mathbb{S}_{n-1}$ is the $(n-1)$st large (big) Schr\"{o}der number.
\end{theorem}

For the rest of this section, $\mathbb{S}_n$ is the $n$th large (big) Schr\"{o}der number for all $n\in\mathbb{N}$.

We first present our result on $S_{n,r}(3412,3421)$ since it has an easier presentation but still shows the subtle difference between this pair of pattern of length four and \cref{Theorem:LeadingTerms231}. In addition, our result on $S_{n,r}(3412,3421)$ also allows us to provide an alternate proof of a recurrence relation on the large Schr\"{o}der numbers.

\begin{theorem}\label{Theorem:34123421leading}
For all $n\geq2$ and $r\in\{1,2,n\}$, we have
\begin{equation*}
    |S_{n,r}(3412,3421)|=\mathbb{S}_{n-2};
\end{equation*}
and for all $n\geq4$ and $2<r<n$, we have
\begin{equation*}
    |S_{n,r}(3412,3421)|=\mathbb{S}_{r-2}\mathbb{S}_{n-r}.
\end{equation*}    
\end{theorem}
\begin{proof}
    First, suppose $n\geq 1$ and $r\in\{1,n\}$. Let $\tau\in S_{n,r}$. If $r=1$, then $r<a$ for all $a\in\mathcal{D}_\tau(r)$. If $r=n$, then $r>a$ for all $a\in\mathcal{D}_\tau(r)$. If $r=2$, then there is exactly one $a\in \mathcal{D}_\tau(r)$ with $a<r$. In any case, $rxyz$ is not a 3412 pattern or a 3421 pattern for any $x,y,z\in\mathcal{D}_\tau(r)$. Hence $\tau\in S_{n,r}(3412,3421)$ if and only if $(\tau(2),\tau(3),\ldots,\tau(n))$ avoids both $3412$ and $3421$. Therefore, by \cref{Theorem:34123421}, we have $|S_{n,r}(3412,3421)|=|S_{n-1}(3412,3421)|=\mathbb{S}_{n-2}$.

    Now suppose $n\geq4$ and $2<r<n$. Let $\mathcal{R}$ be a subset of $S_{n,r}$ such that every $\tau\in\mathcal{R}$ has the following properties:
    \begin{itemize}
        \item[(i)] $\{\tau(2),\tau(3),\ldots,\tau(r-1)\}\subseteq\{1,2,\ldots,r-1\}$;
        \item[(ii)] the subpermutation $\tau'$ of $\tau$ on $\{1,2,\ldots,r-1\}$ avoids both $3412$ and $3421$;
        \item[(iii)] $\tau''=(\tau(r),\tau(r+1),\ldots,\tau(n))$ avoids both $3412$ and $3421$.
    \end{itemize}
Let $\tau\in \mathcal{R}$. By \cref{Theorem:34123421}, there are $\mathbb{S}_{r-2}$ ways for $\tau'$ to avoid both $3412$ and $3421$, and for each fixed $\tau'$, there are $\mathbb{S}_{n-r}$ ways for $\tau''$ to avoid both $3412$ and $3421$. Hence, we have $|\mathcal{R}|=\mathbb{S}_{r-2}\mathbb{S}_{n-r}$.

    Now we show that $S_{n,r}(3412,3421)=\mathcal{R}$. Let $\tau\in S_{n,r}(3412,3421)$, $\tau'$ the subpermutation of $\tau$ on $\{1,2,\ldots,r-1\}$, and $\tau''=(\tau(r),\tau(r+1),\ldots,\tau(n))$. Since $\tau$ avoids both $3412$ and $3421$, $\tau'$ 
avoids both $3412$ and $3421$ as well. Similarly, $\tau''$ avoids both $3412$ and $3421$.

We now show that $\tau(i)\in\{1,2,\ldots,r-1\}$ for all $i\in\{2,3,\ldots,r-1\}$. Suppose, by way of contradiction, that $\tau(i)>r$ for some $i\in\{2,3,\ldots,r-1\}$. Then, since $\tau(1)=r$, at most $r-3$ numbers in $\{\tau(1),\tau(2),\ldots,\tau(r-1)\}$ are less than $r$. So there exist $k>j>r-1$ such that $\tau(j),\tau(k)<r$. Now $r\tau(i)\tau(j)\tau(k)$ is either a 3412 pattern or a 3421 pattern. This is a contradiction. Hence, we have $S_{n,r}(3412,3421)\subseteq\mathcal{R}$.

    
    On the other hand, suppose $\tau\in \mathcal{R}$. We will show that $\tau\in S_{n,r}(3412,3421)$. Suppose, by way of contradiction, that $xyzw$ is a subpermutation of $\tau$ which is a $3412$ pattern or a $3421$ pattern. Then we have $z, w<x<y$ and $z,w\in D_\tau(y)$. We split into three cases:

    Case 1: $x=r$. Then $y>r$. Since $\{\tau(2),\tau(3),\ldots,\tau(r-1)\} \subseteq \{1,2,\ldots r-1\}$, we must have $y=\tau(i)$ for some $i>r-1$ and at most one $j>i$ with $\tau(j)<r$. So either $z>r=x$ or $w>r=x$ which is a contradiction.

    Case 2: $x<r$. Since the subpermutation on $\{1,2,\ldots,r-1\}$ avoids both $3412$ and $3421$, we must have $y>r$. The rest of the argument is then the same as Case 1.

    Case 3: $x>r$. Since $\{\tau(2),\tau(3),\ldots,\tau(r-1)\}\subseteq\{1,2,\ldots,r-1\}$, $xyzw$ is a subpermutation of $(\tau(r),\tau(r+1),\ldots,\tau(n))$. Since $(\tau(r),\tau(r+1),\ldots,\tau(n))$ avoids both $3412$ and $3421$, $xyzw$ is not a $3412$ or $3421$ pattern. This is a contradiction.

    This completes the proof that $\mathcal{R}\subseteq S_{n,r}(3412,3421)$.

    Hence we have $S_{n,r}(3412,3421)=\mathcal{R}$, and therefore \[|S_{n,r}(3412,3421)|=|\mathcal{R}|=\mathbb{S}_{r-2}\mathbb{S}_{n-r}.\qedhere\]
\end{proof}
Summing over $r$ in \cref{Theorem:34123421leading}, we have the following recurrence relation for $\mathbb{S}_n$:

\begin{corollary}\label{Corollary:RecurrenceSchroder}
For all $n\geq1$,
\[
\mathbb{S}_{n+1}=\mathbb{S}_n+\sum_{r=0}^{n}\mathbb{S}_{r}\mathbb{S}_{n-r}.
\]
\end{corollary}

\begin{proof}
    Let $n\geq1$. Note that by \cref{first_table}, we have $\mathbb{S}_0=1$. So by \cref{Theorem:34123421leading}, we have $|S_{n+2,2}(3412,3421)|=|S_{n+2,n+2}(3412,3421)|=\mathbb{S}_n=\mathbb{S}_n\mathbb{S}_0$. Now, by \cref{Theorem:34123421,Theorem:34123421leading}, we have
    \[
    \mathbb{S}_{n+1}=|S_{n+2}(3412,3421)|=\sum_{r=1}^{n+2}|S_{n+2,r}(3412,3421)|=\mathbb{S}_{n}+\sum_{r=2}^{n+2}\mathbb{S}_{r-2}\mathbb{S}_{n+2-r}=\mathbb{S}_n+\sum_{r=0}^{n}\mathbb{S}_{r}\mathbb{S}_{n-r}.\qedhere
    \]
\end{proof}

\begin{remark}
Qi and Guo \cite[Theorem 5]{QiGuo2017} proved \cref{Corollary:RecurrenceSchroder} using generating functions. In \cite[p. 446]{Bona2022}, it is also noted that \cref{Corollary:RecurrenceSchroder} can also be 
derived from the recurrence $\mathbb{S}_n=\sum_{i=0}^n\binom{2n-i}{i}C_{n-i}$ which was proved by West \cite[p. 255]{West1996}. Our proof of this identity does not use the Catalan numbers and is purely combinatorial. 
\end{remark}

Now we generalize our result for $S_{n,r}(3412,3421)$ to $S_{n,(c_1,c_2,\ldots,c_t)}(3412,3421)$. As before, we assume that $(c_1,c_2,\ldots,c_t)$ avoids both $3412$ and $3421$. 

\begin{theorem}
    Let
    \[
    U=\{c_i:i\in[t]~\text{and there exist}~j,k\in[t]~\text{such that}~i<j<k~\text{and}~c_ic_jc_k~\text{is a 231 pattern}\}
    \]
    and
    \[
    V=\{c_i:i\in[t]~\text{and there exists}~j\in[t]~\text{such that}~i<j~\text{and}~c_i<c_j\}.
    \]
    If $U\neq\emptyset$ and $\left|[\max U]\backslash\{c_1,c_2,\ldots,c_t\}\right|\geq1$ or $V\neq\emptyset$ and $\left|[\max V]\backslash\{c_1,c_2,\ldots,c_t\}\right|\geq2$, then $|S_{n,(c_1,c_2,\ldots,c_t)}(3412,3421)|=0$; otherwise,
\[
|S_{n,(c_1,c_2,\ldots,c_t)}(3412,3421)|=\mathbb{S}_{c_{(j)}-c_{(j-1)}-2}\prod_{i=j+1}^{t+1}\mathbb{S}_{c_{(i)}-c_{(i-1)}-1},
\]
where $c_{(0)}=0$, $c_{(t+1)}=n+1$, $c_{(1)}<c_{(2)}<\cdots<c_{(t)}$ are the order statistics of $\{c_1,c_2,\ldots,c_t\}$, and $j=\min\{i\in[t+1]:c_{(i)}-c_{(i-1)}>1\}$.
\end{theorem}
\begin{proof}
    For all $k\in[t+1]$, let $A_k=\{c_{(k-1)}+1,c_{(k-1)}+2,\ldots,c_{(k)}-1\}$. We note that it is possible that $A_k=\emptyset$ for some $k$. Also notice that
    \[
[n]\backslash\{c_1,c_2,\ldots,c_t\}=\bigcup_{k=1}^{t+1}A_k.
\]

We first suppose that $U\neq\emptyset$ and $\left|[\max U]\backslash\{c_1,c_2,\ldots,c_t\}\right|\geq1$. Let $x,y,z\in[t]$ such that $c_x=\max U$, $x<y<z$, and $c_xc_yc_z$ is a $231$ pattern. Since $\left|[c_x]\backslash\{c_1,c_2,\ldots,c_t\}\right|\geq1$, there exists $\alpha\in [n]\backslash\{c_1,c_2,\ldots,c_t\}$ such that $\alpha<c_x$. If $\alpha<c_z$, then $c_xc_yc_z\alpha$ is a $3421$ pattern; and if $\alpha>c_z$, then $c_xc_yc_z\alpha$ is a $3412$ pattern. Hence $|S_{n,(c_1,c_2,\ldots,c_t)}(3412,3421)|=0$.

Next, we suppose that $V\neq\emptyset$ and $\left|[\max V]\backslash\{c_1,c_2,\ldots,c_t\}\right|\geq2$. Let $x,y\in[t]$ such that $c_x=\max V$, $x<y$, and $c_x<c_y$. Since $\left|[c_x]\backslash\{c_1,c_2,\ldots,c_t\}\right|\geq2$, there exist $\alpha,\beta\in[n]\backslash\{c_1,c_2,\ldots,c_t\}$ such that $\alpha<\beta<c_x<c_y$. If $\alpha\in\mathcal{A}_\tau(\beta)$, then $c_xc_y\alpha\beta$ is a $3412$ pattern; and if $\alpha\in\mathcal{D}_\tau(\beta)$, then $c_xc_y\beta\alpha$ is a $3421$ pattern. Hence $|S_{n,(c_1,c_2,\ldots,c_t)}(3412,3421)|=0$.

Now suppose otherwise. Let $j=\min\{i\in[t+1]:c_{(i)}-c_{(i-1)}>1\}$. In other words, $j\in[t+1]$ is the smallest index such that $A_j\neq\emptyset$. For all $\tau\in S_{n,(c_1,c_2,\ldots,c_t)}$, let $\tau_j$ be the subpermutation of $\tau$ on $A_j$, and for all $i\in\{j+1,j+2,\ldots,t+1\}$, let $x_i$ be the last term of the subpmermutation of $\tau$ on $A_j\cup A_{j+1}\cup\cdots \cup A_{i-1}$ and $\tau_i$ be the subpermutation of $\tau$ on $A_i\cup\{x_i\}$.
    
    Let $\mathcal{R}'$ be the subset of $S_{n,(c_1,c_2,\ldots,c_t)}$ such that every $\tau\in\mathcal{R}'$ satisfies the following:
    \begin{itemize}
        \item[(i)] for all $i\in\{j+1,j+2,\ldots,t+1\}$, $y\in A_i$, and $z\in (A_j\cup A_{j+1}\cup\cdots \cup A_{i-1})\backslash\{x_i\}$, we have $z\in\mathcal{A}_\tau(y)$;
        \item[(ii)] and for all $i\in\{j,j+1,\ldots,t+1\}$, $\tau_i$ avoids both $3412$ and $3421$.
    \end{itemize}
    
    Let $\tau\in\mathcal{R}'$. By \cref{Theorem:34123421}, since $|A_j|=c_{(j)}-c_{(j-1)}-1$, there are $\mathbb{S}_{c_{(j)}-c_{(j-1)}-2}$ possibilities for $\tau_j$. Now let $i\in\{j+1,j+2,\ldots,t+1\}$. Inductively, we count the possibilities for $\tau_i$ when $\tau_j,\tau_{j+1},\ldots,\tau_{i-1}$ are fixed. In this case, since $|A_i\cup\{x_i\}|=c_{(i)}-c_{(i-1)}$, by \cref{Theorem:34123421}, there are $\mathbb{S}_{c_{(i)}-c_{(i-1)}-1}$ possibilities for $\tau_i$. Hence, we have $|\mathcal{R}'|=\mathbb{S}_{c_{(j)}-c_{(j-1)}-2}\prod_{i=j+1}^{t+1}\mathbb{S}_{c_{(i)}-c_{(i-1)}-1}$.

    It remains to show that $\mathcal{R}'=S_{n,(c_1,c_2,\ldots,c_t)}(3412,3421)$. Let $\tau\in S_{n,(c_1,c_2,\ldots,c_t)}(3412,3421)$. Then Property (ii) is obviously satisfied. Now we show that $\tau$ satisfies Property (i). Suppose, by way of contradiction, that $\tau$ does not satisfy Property (i). Then there exist $i\in\{j+1,j+2,\ldots,t+1\}$, $y\in A_i$, and $z\in (A_j\cup A_{j+1}\cup\cdots \cup A_{i-1})\backslash\{x_i\}$ such that $z\in D_\tau(y)$. Since $x_i$ is the last term of the subpermutation of $\tau$ on $A_j\cup A_{j+1}\cup\cdots\cup A_{i-1}$, $c_{(i-1)}yzx_i$ is a subpermutation of $\tau$. Since $y\in A_i$, we have $y>c_{(i-1)}$. Now since $z,x_i\in A_j\cup A_{j+1}\cup\cdots \cup A_{i-1}$, $c_{(i-1)}yzx_i$ is either a $3412$ pattern or a $3421$ pattern. This is a contradiction. Hence $\tau$ satisfies Property (i). It follows that $S_{n,(c_1,c_2,\ldots,c_t)}(3412,3421)\subseteq\mathcal{R}'$. 

    We still need to prove that $\mathcal{R}'\subseteq S_{n,(c_1,c_2,\ldots,c_t)}(3412,3421)$. Let $\tau\in \mathcal{R}'$. Suppose, by way of contradiction, that $abcd$ is a subpermutation of $\tau$ which is either a $3412$ pattern or a $3421$ pattern. Then we have $c<d<a<b$ or $d<c<a<b$. We have five cases:

    Case 1: $a,b,c,d\notin\{c_1,c_2,\ldots,c_t\}$. Then $a\in A_k$ for some $k\in\{j,j+1,\ldots,t+1\}$. Since $c,d<a$, we have $c\in A_{i_1}$ and $d\in A_{i_2}$ for some $i_1,i_2\leq k$. Since $b>a$, if $b\notin A_k$, then $b\in A_{i_3}$ with $i_3>k$. Then Property (i) is violated because $c,d\in A_j\cup A_{j+1}\cup\cdots\cup A_{k}$. So we must have $b\in A_k$. By Property (i), at most one of $c$ and $d$ is in $A_j\cup A_{j+1}\cup\cdots\cup A_{k-1}$. In addition, if $c$ or $d$ is in $A_j\cup A_{j+1}\cup\cdots\cup A_{k-1}$, then it must be the last term of $\tau_{k-1}$. So $abcd$ is a subpermutation on $A_k\cup\{x_k\}$. This contradicts Property (ii).
    
    Case 2: $a\in\{c_1,c_2,\ldots,c_t\}$ but $b,c,d\notin\{c_1,c_2,\ldots,c_t\}$. Then $a=c_k$ for some $k\in[t]$. Since $b>a$, we have $b\in A_{i}$ for some $i>k$. Since $c,d<a$, we have $c,d\in A_j\cup A_{j+1}\cup\cdots \cup A_{k}$. Since $c,d\in D_\tau(b)$, this violates Property (i). Hence we have a contradiction.

    Case 3: $a,b\in\{c_1,c_2,\ldots,c_t\}$ but $c,d\notin\{c_1,c_2,\ldots,c_t\}$. Since $a<b$ and $c,d<a$, $V\neq\emptyset$ and $\left|[\max V]\backslash\{c_1,c_2,\ldots,c_t\}\right|\geq2$ which is a contradiction.

    Case 4: $a,b,c\in\{c_1,c_2,\ldots,c_t\}$ but $d\notin\{c_1,c_2,\ldots,c_t\}$. Then $abc$ is a $231$ pattern. Since $d<a$, $U\neq\emptyset$ and $\left|[\max U]\backslash\{c_1,c_2,\ldots,c_t\}\right|\geq1$ which is a contradiction.

    Case 5: $a,b,c,d\in\{c_1,c_2,\ldots,c_t\}$. This contradicts our convention that $(c_1,c_2,\ldots,c_t)$ avoids both $3412$ and $3421$.

    Hence, $\tau$ avoids both $3412$ and $3421$. It follows that $\mathcal{R}'\subseteq S_{n,(c_1,c_2,\ldots,c_t)}(3412,3421)$. This completes our proof.
\end{proof}

\section{$r$-Wilf-equivalence classes}\label{Section:LeadingWilf}
In this section, we classify $r$-Wilf-equivalence classes for patterns of length three. Recall that for a fixed $r\in\mathbb{N}$, two patterns $\sigma$ and $\sigma'$ are 
called {\em $r$-Wilf equivalent} if $|S_{n,r}(\sigma)|=|S_{n,r}(\sigma')|$ for all 
$n\geq r$.

We start with some elementary results summarized in \cref{Table:SingleClassical}.
\begin{table}[H]
\centering
\begin{tabular}{|c||c|c|c|c|c|c|}
\hline
$r$&$|S_{n,r}(123)|$&$|S_{n,r}(321)|$&$|S_{n,r}(132)|$&$|S_{n,r}(312)|$&$|S_{n,r}(213)|$&$|S_{n,r}(231)|$\\
\hline\hline
$n$&$C_{n-1}$&$1$&$C_{n-1}$&$1$&$C_{n-1}$&$C_{n-1}$\\
\hline
$n-1$&$C_{n-1}$&$n-1$&$C_{n-1}$&$n-1$&$C_{n-2}$&$C_{n-2}$\\\hline
$2$&$n-1$&$C_{n-1}$&$n-1$&$C_{n-1}$&$C_{n-2}$&$C_{n-2}$\\\hline
$1$&$1$&$C_{n-1}$&$1$&$C_{n-1}$&$C_{n-1}$&$C_{n-1}$\\\hline
\end{tabular}
\caption{Single Patterns of Length $3$ for $n\geq2$.}\label{Table:SingleClassical}
\end{table}

It is easy to check the correctness of the expressions in \cref{Table:SingleClassical}. 
As an example, we sketch the proof of the fact 
that $|S_{n,n-1}(123)|=C_{n-1}$ for all $n\geq2$. For any $i,j\in\{2,3,\ldots,n\}$ with $i<j$, either $\tau(i)<n-1$ or $\tau(j)<n-1$. It follows that $(n-1,\tau(i),\tau(j))$ 
will never form a $123$ pattern for any $i,j\in\{2,3,\ldots,n\}$ with $i<j$. Hence, $\tau\in S_{n,n-1}$ avoids $123$ if and only if $(\tau(2),\tau(3),\ldots,\tau(n))$ avoids $123$. Therefore, by \cref{Theorem:ClassicalSingle3}, we have $|S_{n,n-1}(123)|=|S_{n-1}(123)|=C_{n-1}$.

Next we classify 
the $r$-Wilf-equivalence classes for patterns of length three
for all $r\in\mathbb{N}$.
\begin{theorem}\label{Theorem:rWilf3}
There are two $1$-Wilf-equivalence classes for patterns of length three: $123\stackrel{1}{\sim} 132$ and $321\stackrel{1}{\sim}312\stackrel{1}{\sim}213\stackrel{1}{\sim}231$. 
For $r\geq2$, there are three $r$-Wilf-equivalence classes for patterns of length three: $213\rsim231$, $123\rsim132$, and $321\rsim312$. 
\end{theorem}
\begin{proof}
The fact that there are two $1$-Wilf-equivalence classes for patterns of length three 
follows from the last row of \cref{Table:SingleClassical}.

Let $r\geq2$. By \cref{lemma:complements}, we have $213\rsim231$, $123\rsim132$, and $321\rsim312$. We need to show that there exist $n_1,n_2,n_3\geq r$ such that $|S_{n_1,r}(213)|\neq|S_{n_1,r}(123)|$, $|S_{n_2,r}(213)|\neq|S_{n_2,r}(321)|$, and $|S_{n_3,r}(123)|\neq|S_{n_3,r}(321)|$. There
are three cases to consider. 

Case 1: $r=2$. Set $n_1=n_2=n_3=4$. By \cref{first_table,Table:SingleClassical}, we have $|S_{4,2}(123)|=4-1=3$, $|S_{4,2}(321)|=C_3=5$, and $|S_{4,2}(213)|=C_2=2$. Hence we have the desired result.

Case 2: $r=3$. Set $n_1=n_2=n_3=4$. By \cref{first_table,Table:SingleClassical}, we have $|S_{4,3}(123)|=C_4=5$, $|S_{4,3}(321)|=4-1=3$, and $|S_{4,3}(213)|=C_2=2$. Hence we have the desired result.

Case 3: $r\geq3$. Set $n_1=n_2=n_3=r+1$. By \cref{Table:SingleClassical}, we have 
$|S_{r+1,r}(123)|=C_r$, $|S_{r+1,r}(321)|=r$, and $|S_{r+1,r}(213)|=C_{r-1}$. 
By \cref{first_table}, we have $C_r>C_{r-1}>r$ and the theorem follows.
\end{proof}

In addition to the classical patterns we have studied so far in this paper, many papers studied consecutive patterns \cite{ElizaldeNoy2003}, bivincular patterns \cite{BCDK2010}, and mesh patterns \cite{BrandenClaesson2011,HJSVU2015}. Here we briefly describe, for all $r\geq5$, the $r$-Wilf equivalence classes for vincular patterns of length three studied by Babson and Steingrímsson \cite{BabsonSteingrimsson2000} and later, Claesson \cite{Claesson2001}. 

In vincular patterns \cite[Section 2]{HJSVU2015}, some consecutive elements in a permutation pattern are required to be adjacent. We use overlines to indicate that the elements under the overlines are required to be adjacent. There are 
twelve vincular patterns of length three where one requires exactly two numbers to be adjacent. For example, a permutation $\tau\in S_n$ contains the pattern $1\overline{32}$ if there exist indices $i<j$ such that $\tau(i)\tau(j)\tau(j+1)$ is a $132$ pattern. Other vincular patterns are defined similarly.

\begin{example}
    In the permutation $\tau=13542\in S_5$, $\tau(2)\tau(3)\tau(5)=352$ is a $\overline{23}1$ pattern and $\tau(1)\tau(4)\tau(5)=142$ is a $1\overline{32}$ pattern, but $\tau$ avoids the pattern $\overline{21}3$.
\end{example}
Claesson \cite[Propositions 1-3 and Lemma 2]{Claesson2001} proved that there are two Wilf-equivalence classes for the 
twelve vincular patterns. They are counted either by the Catalan numbers or by the Bell numbers:

\begin{theorem}\label{Theorem:GeneralizedSingle3}
For all $n\geq1$,
\[
\begin{split}
|S_n(1\overline{23})|=&|S_n(3\overline{21})|=|S_n(\overline{12}3)|=|S_n(\overline{32}1)|=|S_n(1\overline{32})|\\=&|S_n(3\overline{12})|=|S_n(\overline{21}3)|=|S_n(\overline{23}1)|=B_n,
\end{split}
\]
and
\[
|S_n(2\overline{13})|=|S_n(2\overline{31})|=|S_n(\overline{13}2)|=|S_n(\overline{31}2)|=C_n,
\]
where $B_n$ is the $n$th Bell number and $C_n$ is the $n$th Catalan number.
\end{theorem}

We first adapt some results in Claesson \cite{Claesson2001} to show $r$-Wilf equivalence for several vincular patterns.

\begin{proposition}\label{Proposition:rWilfClassesWithMoreThan1}
For all $r\in\mathbb{N}$, $2\overline{13}\rsim2\overline{31}$, $1\overline{23}\rsim1\overline{32}$, and $3\overline{21}\rsim3\overline{12}$.
\end{proposition}
\begin{proof}
Let $1\leq r\leq n$.
Using a short combinatorial argument, Claesson \cite[Lemma 2]{Claesson2001} showed that for all $n\geq1$, $\tau\in S_n$ avoids $2\overline{13}$ if and only if it avoids $213$. 
Taking complements, for all $n\geq1$, $\tau\in S_n$ avoids $2\overline{31}$ if and only if it avoids $231$. Hence for all $\tau\in S_{n,r}$, $\tau$ avoids $2\overline{13}$ if and only if it avoids $213$ and $\tau$ avoids $2\overline{31}$ if and only if it avoids $231$. Then we have $|S_{n,r}(2\overline{13})|=|S_{n,r}(213)|$ and $|S_{n,r}(2\overline{31})|=|S_{n,r}(231)|$. Now by \cref{lemma:complements}, we have $|S_{n,r}(2\overline{13})|=|S_{n,r}(213)|=|S_{n,r}(231)|=|S_{n,r}(2\overline{31})|$. Therefore, $2\overline{13}\rsim2\overline{31}$.

Claesson \cite[Propositions 2 and 4]{Claesson2001} constructed bijections between $S_n(1\overline{23})$ and the partitions of $[n]$, and then between $S_n(1\overline{32})$ and the partitions of $[n]$. These bijections preserve 
the leading terms of permutations. So for all $1\leq r\leq n$, we have $|S_{n,r}(1\overline{23})|=|S_{n,r}(1\overline{32})|$. Taking the complements, we also have $|S_{n,r}(3\overline{21})|=|S_{n,r}(3\overline{12})|$. Therefore, we have $1\overline{23}\rsim1\overline{32}$ and $3\overline{21}\rsim3\overline{12}$.
\end{proof}


By \cref{Proposition:rWilfClassesWithMoreThan1}, there are at most nine $r$-Wilf-equivalence classes for vincular patterns. \cref{Table:GeneralizedPatterns} lists the results we need to classify $r$-Wilf equivalence classes for all twelve vincular patterns. 

\begin{table}[h]
\centering
\renewcommand{\arraystretch}{1.2}
\begin{tabular}{|c||c|c|c|c|c|}
\hline
&$n=r$&$n=r+1$&$n=r+2$\\\hline\hline
$|S_{n,r}(2\overline{13})|=|S_{n,r}(2\overline{31})|$&$C_{r-1}$&$C_{r-1}$&\\\hline
$|S_{n,r}(\overline{13}2)|$&$C_{r-1}$&$C_r$&\\\hline
$|S_{n,r}(3\overline{21})|=|S_{n,r}(3\overline{12})|$&$1$&$2^{r-1}$&\\\hline
$|S_{n,r}(\overline{31}2)|$&$1$&$r$&\\\hline
$|S_{n,r}(1\overline{23})|=|S_{n,r}(1\overline{32})|$&$B_{r-1}$&$B_r$&$B_{r+1}-B_{r-1}$\\\hline
$|S_{n,r}(\overline{12}3)|$&$B_{r-1}$&$B_r$&$B_{r+1}-B_r$\\\hline
$|S_{n,r}(\overline{21}3)|$&$B_{r-1}$&$B_{r-1}$&\\\hline
$|S_{n,r}(\overline{23}1)|$&$B_{r-1}$&$B_r-B_{r-1}$&\\\hline
$|S_{n,r}(\overline{32}1)|$&$B_{r-2}$&&\\\hline
\end{tabular}
\caption{Avoiding Vincular Patterns by Leading Terms for $r\geq3$. (We leave some entries in the table blank and only include results that are needed to classify $r$-Wilf equivalence classes for the twelve vincular patterns.)}\label{Table:GeneralizedPatterns}
\end{table}

Most of the expressions in \cref{Table:GeneralizedPatterns} can be obtained by straightforward calculation 
using \cref{Theorem:GeneralizedSingle3} and \cref{Table:SingleClassical}. We will 
only sketch the proofs of a few of them.

\begin{lemma}
For all $r\geq3$, \[|S_{r,r}(\overline{32}1)|=B_{r-2}.\]
\end{lemma}
\begin{proof}
Let $\tau\in S_{r,r}(\overline{32}1)$. If $\tau(2)\neq 1$, then $\tau(1)\tau(2)1$ is a $\overline{32}1$ pattern. So we must have $\tau(2)=1$. At the same time, $(\tau(3),\tau(4),\ldots,\tau(n))$ avoids the pattern $\overline{32}1$. So by \cref{Theorem:GeneralizedSingle3}, we have $|S_{r,r}(\overline{32}1)|\leq|S_{r-2}(\overline{32}1)|=B_{r-2}$.

Now let $\tau\in S_{r,r}$ with $\tau(2)=1$ and $(\tau(3),\tau(4),\ldots,\tau(n))$ avoiding the pattern $\overline{32}1$. Since $r1x$ and $1xy$ are never $\overline{32}1$ patterns, we must have $\tau\in S_{r,r}(\overline{32}1)$. Hence we have $B_{r-2}=|S_{r-2}(\overline{32}1)|\leq |S_{r,r}(\overline{32}1)|$. This completes the proof of the lemma.
\end{proof}

\begin{lemma}
For all $r\geq1$, \[|S_{r+2,r}(1\overline{23})|=B_{r+1}-B_{r-1} \text{ and } |S_{r+2,r}(\overline{12}3)|=B_{r+1}-B_{r}.\]
\end{lemma}
\begin{proof}
We first prove that $|S_{r+2,r}(1\overline{23})|=B_{r+1}-B_{r-1}$. Let $\tau\in S_{r+2,r}(1\overline{23})$. Then the subpermutation $(\tau(2),\tau(3),\ldots,\tau(r+2))$ avoids $1\overline{23}$. By \cref{Theorem:GeneralizedSingle3}, there are $|S_{r+1}(1\overline{23})|=B_{r+1}$ ways for $(\tau(2),\tau(3),\ldots,\tau(r+2))$ to avoid $1\overline{23}$. For these permutations on $\{1,2,\ldots,r-1,r+1,r+2\}$, the only way that $r+1$ and $r+2$ are adjacent and the subpermutation on $\{r+1,r+2\}$ is $(r+1,r+2)$ is when $\tau(2)=r+1$ and $\tau(3)=r+2$ because otherwise $(\tau(2),\tau(3),\ldots,\tau(r+2))$ would contain a $1\overline{23}$ pattern. This is the only case that $(r,r+1,r+2)$ is a $1\overline{23}$ pattern. Since $(\tau(4),\tau(5),\ldots,\tau(r+2))$ also need to avoid $1\overline{23}$, by \cref{Theorem:GeneralizedSingle3}, the number of permutations $(\tau(2),\tau(3),\ldots,\tau(r+2))$ avoiding $1\overline{23}$, with $\tau(2)=r+1$ and $\tau(3)=r+2$, is $|S_{r-1}(1\overline{23})|=B_{r-1}$. Here it is easy to check that if $\tau(2)=r+1$, $\tau(3)=r+2$, and $(\tau(4),\tau(5),\ldots,\tau(r+2))$ avoids $1\overline{23}$, then $(\tau(2),\tau(3),\ldots,\tau(r+2))$ avoids $1\overline{23}$ as well. 
Therefore
\[
|S_{r+2,r}(1\overline{23})|=|S_{r+1}(1\overline{23})|-|S_{r-1}(1\overline{23})|=B_{r+1}-B_{r-1}.
\]

Next, we prove that $|S_{r+2,r}(\overline{12}3)|=B_{r+1}-B_r$. Let $\tau\in S_{r+2,r}(\overline{12}3)$. Then $(\tau(2),\tau(3),\ldots,\tau(r+2))$ avoids $\overline{12}3$. By \cref{Theorem:GeneralizedSingle3}, there are $|S_{r+1}(\overline{12}3)|=B_{r+1}$ ways for $(\tau(2),\tau(3),\ldots,\tau(r+2))$ to avoid $\overline{12}3$. For these permutations, the only way that we have a $\overline{12}3$ pattern starting with $r$ is when $\tau(2)=r+1$, then $(\tau(1),\tau(2),r+2)$ is a $\overline{12}3$ pattern. Here, it is easy to see that if $\tau(2)=r+1$, then, for all $2<i<j\leq r+2$, $(\tau(2),\tau(i),\tau(j))$ is never a $\overline{12}3$ pattern. Hence, by \cref{Theorem:GeneralizedSingle3}, the number of permutations $(\tau(2),\tau(3),\ldots,\tau(r+2))$, with $\tau(2)=r+1$, avoiding $\overline{12}3$ is $|S_{r}(\overline{12}3)|=B_r$. Using subtraction, we have
\[
|S_{r+2,r}(\overline{12}3)|=|S_{r+1}(\overline{12}3)|-|S_{r}(\overline{12}3)|=B_{r+1}-B_{r}.\qedhere
\]
\end{proof}
\begin{lemma}
For all $r\geq1$, \[|S_{r+1,r}(3\overline{21})|=2^{r-1}.\]
\end{lemma}
\begin{proof}
Let $\tau\in S_{r+1,r}(3\overline{21})$ and let $i>1$ be such that $\tau(i)=r+1$. Then for all $j\in\{2,3,\ldots,i-2\}$, we must have $\tau(j)<\tau(j+1)$. To see this, suppose that $\tau(j)>\tau(j+1)$ for some $j\in\{2,3,\ldots,i-2\}$. Then $r\tau(j)\tau(j+1)$ is a $3\overline{21}$ pattern which is a contradiction. Similarly, for all $j\in\{i+1,i+2,\ldots,n-1\}$, we must have $\tau(j)<\tau(j+1)$. Hence, we have $\tau(2)<\tau(3)<\cdots<\tau(i-1)$ and $\tau(i+1)<\tau(i+2)<\cdots<\tau(r+1)$. 

On the other hand, it is easy to check that for all $\tau\in S_{r+1,r}$, if $\tau(i)=r+1$, $\tau(2)<\tau(3)<\cdots<\tau(i-1)$ , and $\tau(i+1)<\tau(i+2)<\cdots<\tau(r+1)$ for some $i>1$, then $\tau$ avoids $3\overline{21}$. 

So $|S_{r+1,r}(3\overline{21})|$ is equal to the number of permutations $\tau\in S_{r+1,r}$ such that for some $i\in\{2,3,\ldots,r+1\}$, we have $\tau(i)=r+1$, $\tau(2)<\tau(3)<\cdots<\tau(i-1)$, and $\tau(i+1)<\tau(i+2)<\cdots<\tau(r+1)$. Let $\tau$ be such a permutation and $i\in\{2,3,\ldots,r+1\}$. Then there are $\binom{r-1}{i-2}$ ways to choose $i-2$ numbers from $\{1,2,\ldots,r-1\}$ and assign them to $\tau(2),\tau(3),\ldots,\tau(i-1)$ so that $\tau(2)<\tau(3)<\cdots<\tau(i-1)$; once $\tau(2),\tau(3),\ldots,\tau(i-1)$ are determined, $\tau(i+1),\tau(i+2),\ldots,\tau(r+1)$ are uniquely determined as well. Hence we have
\[
|S_{r+1,r}(3\overline{21})|=\sum_{i=2}^{r+1}\binom{r-1}{i-2}=\sum_{i=0}^{r-1}\binom{r-1}{i}=2^{r-1}.\qedhere
\]
\end{proof}

If $r\geq5$, by \cref{first_table,Table:GeneralizedPatterns} and \cref{Lemma:CatalanBell,Lemma:Bell}, there are nine $r$-Wilf equivalence classes. To see this, it suffices to note that for each $r\geq5$ and for any two distinct generalized patterns $\sigma$ and $\sigma'$ in different rows, either $|S_{r,r}(\sigma)|\neq|S_{r,r}(\sigma')|$, or $|S_{r+1,r}(\sigma)|\neq|S_{r+1,r}(\sigma')|$, or $|S_{r+2,r}(\sigma)|\neq|S_{r+2,r}(\sigma')|$. We briefly describe several of them as examples.
\begin{example}
    By \cref{Table:GeneralizedPatterns} and \cref{Lemma:CatalanBell}, for all $r\geq5$, $|S_{r,r}(\overline{13}2)|=C_{r-1}<B_{r-1}=|S_{r,r}(\overline{12}3)|$. Hence, for all $r\geq5$, $\overline{13}2$ and $\overline{12}3$ are not $r$-Wilf equivalent.
\end{example}
\begin{example}
    By \cref{Table:GeneralizedPatterns} and \cref{Lemma:Bell}, for all $r\geq5$, $|S_{r,r}(\overline{21}3)|=B_{r-1}=|S_{r,r}(\overline{23}1)|$, but $|S_{r+1,r}(\overline{21}3)|=B_{r-1}<B_r-B_{r-1}=|S_{r+1,r}(\overline{23}1)|$. Hence, for all $r\geq5$, $\overline{21}3$ and $\overline{23}1$ are not $r$-Wilf equivalent.
\end{example}

\begin{example} By \cref{Table:GeneralizedPatterns}, we have $|S_{r,r}(1\overline{23})|=B_{r-1}=|S_{r,r}(\overline{12}3)|$ and $|S_{r+1,r}(1\overline{23})|=B_{r}=|S_{r+1,r}(\overline{12}3)|$, but $|S_{r+2,r}(1\overline{23})|=B_{r+1}-B_{r-1}>B_{r+1}-B_r=|S_{r+2,r}(\overline{12}3)|$ for all $r\geq5$. Hence $1\overline{23}$ and $\overline{12}3$ belong to two different equivalence classes when $r\geq5$.
\end{example}

\begin{example}
By \cref{Table:GeneralizedPatterns}, we have $|S_{r,r}(2\overline{13})|=C_{r-1}$ and $|S_{r,r}(\overline{32}1)|=B_{r-2}$ for all $r\geq5$. By \cref{first_table} and the generating functions of the Catalan and Bell numbers \cite[Sections 3.2 and 6.1]{EgeciogluGarsia2021}, we have $B_{r-1}\neq C_r$ for all $r\geq5$. Hence $2\overline{13}$ and $\overline{32}1$ belong to two different equivalence classes when $r\geq5$.
\end{example}

The following theorem completely classifies, for all $r\geq5$, the
$r$-Wilf-equivalence classes for the twelve vincular patterns of length three.
\begin{theorem}
    For all $r\geq5$, there are nine $r$-Wilf-equivalence classes for 
vincular patterns of length three: 
$2\overline{13}\rsim 2\overline{31}$, $1\overline{23}\rsim1\overline{32}$, $3\overline{21}\rsim 3\overline{12}$, and the other six classes each contains a single vincular pattern.
\end{theorem}

\section{Concluding remarks}\label{Section:Concluding}
Miner and Pak \cite{MinerPak2014} used generalizations of \cref{Theorem:ClassicalLeading123AND132} to study the limit shapes of random permutations avoiding a given pattern. In this section, we give some ideas about the limit shapes of random $\sigma$-avoiding, $\sigma\in S_3$, permutations with 
fixed prefix $(c_1,c_2,\ldots,c_t)$. Particularly, we are interested in exploring for large $n$, what a uniformly random permutation from $S_{(c_1,c_2,\ldots,c_t)}(\sigma)$ looks like. To do this, we follow Miner and Pak \cite{MinerPak2014} and 
view permutations as matrices. That is, for each $\tau\in S_n$, we look at the $n\times n$ matrix $M(\tau)$ such that $(M(\tau))_{jk}=1$ if $\tau(j)=k$ and $(M(\tau))_{jk}=0$ if $\tau(j)\neq k$. By \cref{lemma:complements}, complementary patterns may be studied in pairs and it suffices to examine permutations avoiding the patterns $123$, $132$, and $231$.

In some situations, this question is easy to answer. 
If $1\in\{c_1,c_2,\ldots,c_t\}$, then there is a
unique permutation that avoids a $123$ pattern, as the later $n-t$ digits need to be decreasing to avoid a $23$ pattern in the second unfixed segment. So after asymptotic scaling, the limit of the unfixed segment is just the anti-diagonal $x+y=1$. The situation becomes more complicated when $1\notin\{c_1,c_2,\ldots,c_t\}$. As shown in \cref{Theorem:123LeadingTerms}, we may project our permutation from $S_n$ where the first $t$ coordinates are fixed down to a permutation from $S_{n-t+1}$ where only the first coordinate is 
fixed via `standardization.' The limiting phenomenon of these generic `reduced' $123$-avoiding 
permutations were studied in Miner and Pak \cite{MinerPak2014}, where the anti-diagonal again 
shows up. As pointed out earlier in \cref{Section:Single3}, the result for the $132$ 
pattern is similar to the $123$ pattern and the structure of the pattern-avoiding permutation is also preserved 
after projection via standardization. See \cref{Theorem:132LeadingTerms} for more details. The limiting phenomenon 
of these reduced $132$-avoiding permutations was also studied in Miner and Pak \cite{MinerPak2014}, where the anti-diagonal as well as the lower right corner show up in the asymptotic analysis. Unlike $123$ and $132$ patterns, 
fixing the prefix $(c_1,c_2,\ldots,c_t)$, a uniformly random permutation avoiding a $231$ pattern will instead display a block structure as hinted in our proof of \cref{Theorem:LeadingTerms231}. For the initial block which consists of $c_1,c_2,\ldots,c_t$, the segment of the permutation will be a fixed curve that is asymptotically in correspondence to the prefix $(c_1,c_2,\ldots,c_t)$; and for all the remaining blocks, the segment of the permutation will lie on the boundary of feasible $231$-density asymptotically. See Kenyon et al. \cite{Kenyon} and the references therein for a description of the limit shapes of these feasible regions.

We have only scratched the surface of enumerating pattern-avoiding permutations by fixed prefixes, mostly concentrating on patterns of length three. It would be interesting to study permutations with 
fixed prefixes that avoid other patterns; for instance, all single patterns of length greater than three are open. It would also be interesting to compute limits of pattern avoiding permutations chosen uniformly under certain constraints; fixing the prefix $(c_1,c_2,\ldots,c_t)$ as we have done in this paper is only one of the many possibilities out there.
\section*{Acknowledgements}

We would like to thank Darij Grinberg, Stephan Wagner and Lanqing Zhao for helpful conversations, and Mikl\'{o}s B\'{o}na for pointing out the source of \cref{Lemma:AscendIn123Avoiding}. The authors would also like to thank the referee for their suggestions which helped improve the presentation of the paper. \"O.~E\u{g}ecio\u{g}lu 
would like to acknowledge his sabbatical time at Reykjavik University in 2019 during which he 
had a chance to learn about the combinatorics of pattern avoidance. C.~Gaiser and M.~Yin were supported by the University of Denver's Professional 
Research Opportunities for Faculty Fund 80369-145601.

\section*{Declaration of Interests}

The authors declare that they have no known competing financial interests or personal relationships that could have appeared to influence the work reported in this paper.



\end{document}